\documentclass[12pt,a4paper,reqno]{amsart}
\usepackage{amssymb,amsmath}
\usepackage{amsfonts,amsbsy,bm}
\usepackage{latexsym}
\usepackage{exscale}

\usepackage{color}

\usepackage[normalem]{ulem}

\newcommand{\N}{\mathbb N}

\newcommand{\cB}{\mathcal B}

\newcommand{\be}{{\mathbf e}}
\newcommand{\bs}{{\mathbf s}}

\newcommand{\bz}{{\mathbf 0}}

\newcommand{\bx}{{\mathbf x}}
\newcommand{\by}{{\mathbf y}}

\newcommand {\SC} {{\mathbb C}}
\newcommand {\SD} {{\mathbb D}}

\newcommand {\SK} {{\mathbb K}}
\newcommand {\SN} {{\mathbb N}}
\newcommand {\SR} {{\mathbb R}}
\newcommand {\ST} {{\mathbb T}}
\newcommand {\SW} {{\mathbb W}}
\newcommand {\SX} {{\mathbb X}}

\newcommand {\Wd} {{\mathbb{W}_{\rm d}}}
\newcommand {\Wc} {{\mathbb{W}_{\rm co}}}
\newcommand {\Wqc} {{\mathbb{W}_{\rm qc}}}

\newcommand {\al} {{\alpha}}
\newcommand {\dt} {{\delta}}
\newcommand {\Dt} {{\Delta}}
\newcommand {\e} {{\varepsilon}}
\newcommand {\ga} {{\gamma}}

\newcommand {\la} {{\lambda}}

\newcommand {\om} {{\omega}}

\newcommand{\cL}{\mathcal L}
\newcommand{\cY}{\Upsilon} 

\newcommand {\hg} {{\hat g}}

\newcommand{\ba}{{\mathbf a}}

\newcommand{\Been}{{\{\be_n, \be_n^*\}_{n=1}^\infty}}

\newcommand{\LN}{{\mathbf L}_N}
\newcommand{\tLN}{{\mathbf{\widetilde{L}}}_N}

\newcommand{\bDs}{{\mbox{{\small$({\mathbf D}^*)$}}}}

\newcommand{\bdt}{{\boldsymbol\dt}}
\newcommand{\bfe}{{\boldsymbol\e}}

\newcommand{\heta}{{\widehat\eta}}
\newcommand{\teta}{{\widetilde\eta}}
\newcommand{\coeta}{{\eta^\sharp}}
\newcommand{\OT}{{\overline{T}}}
\newcommand{\bKN}{K^{\rm u}_N}
\newcommand{\ttM}{{\mathtt{M}}}
\newcommand{\ld}{{\underline{d}}}

\renewcommand{\phi}{{\varphi}}

\newcommand {\lheta} {{\ell_{\heta}^1}}

\def\supp{\mathop{\rm supp}}
\def\sign{\mathop{\rm sign}}

\numberwithin{equation}{section}
\newtheorem{theorem}{Theorem}[section]
\newtheorem{lemma}[theorem]{Lemma}

\newtheorem{corollary}[theorem]{Corollary}
\newtheorem{Remark}[theorem]{Remark}
\newtheorem{proposition}[theorem]{Proposition}
\newtheorem{definition}[theorem]{Definition}
\newtheorem{example}[theorem]{Example}

\newcommand{\Ba}[1]{\begin{array}{#1}}
\newcommand{\Ea}{\end{array}}
\newcommand{\Be}{\begin{equation}}
\newcommand{\Ee}{\end{equation}}
\newcommand{\Bea}{\begin{eqnarray}}
\newcommand{\Eea}{\end{eqnarray}}
\newcommand{\Beas}{\begin{eqnarray*}}
\newcommand{\Eeas}{\end{eqnarray*}}
\newcommand{\Benu}{\begin{enumerate}}
\newcommand{\Eenu}{\end{enumerate}}
\newcommand{\Bi}{\begin{itemize}}
\newcommand{\Ei}{\end{itemize}}

\newcommand{\BR}{\begin{Remark} \em}
\newcommand{\ER}{\end{Remark}}
\newcommand{\BE}{\begin{example} \em}
\newcommand{\EE}{\end{example}}

\newcommand {\Ds} {\displaystyle}
\newcommand {\Ts} {\textstyle}

\newcommand {\mand} {{\quad\mbox{and}\quad}}
\renewcommand {\mid} {{\,\,\,\colon\,\,\,}}

\newcommand{\bline}{{\bigskip

\noindent}}

\newcommand{\sline}{{\smallskip

\noindent}}

\newcommand{\bp}{{b_p }}
\newcommand{\lpr}{{{\ell}^{p,r}}}

\newcommand{\tri}[1]{|\!|\!| #1 |\!|\!|}

\newcommand {\bone} {{\bf 1}}

\newcommand {\Cspan}[1] {\overline{\mbox{\rm span}}\,\{#1\}}

\newcounter{reg}
\newcounter{regTO}
\newcommand {\Proof} {\noindent{\bf P{\footnotesize\bf ROOF}: } \ }
\newcommand {\Proofof}[1] {\noindent{\bf P{\footnotesize\bf ROOF} of {#1}: } \ }
\newcommand {\ProofEnd} {
             \begin{flushright} \vskip -0.2in $\Box$ \end{flushright}}

\newcommand{\vertiii}[1]{{\left\vert\kern-0.25ex\left\vert\kern-0.25ex\left\vert #1
		\right\vert\kern-0.25ex\right\vert\kern-0.25ex\right\vert}}

\begin{document}

\title[Embeddings and Lebesgue inequalities for greedy algorithms]{Embeddings and Lebesgue-type
       inequalities for the greedy algorithm in Banach spaces}

\author{P.M. Bern\'a}

\address{Pablo M. Bern\'a
\\
Departamento de Matem\'aticas \\
Universidad de Murcia\\
30100 Murcia, Spain} \email{pmanuel.berna@um.es}

\author{O. Blasco}

\address{Oscar Blasco
	\\
	Department of Analysis Mathematics \\ Universidad de Valencia, Campus de
	Burjassot\\ Valencia, 46100, Spain}
	 \email{oscar.blasco@uv.es}

\author{G. Garrig\'os}
\address{Gustavo Garrig\'os
\\
Departamento de Matem\'aticas \\
Universidad de Murcia\\
30100 Murcia, Spain} \email{gustavo.garrigos@um.es}

\author{E. Hern\'andez}

\address{Eugenio Hern\'andez
\\
Departamento de Matem\'aticas\\
Universidad Aut\'onoma de Madrid\\
28049, Madrid, Spain} \email{eugenio.hernandez@uam.es}

\author{T. Oikhberg}

\address{Timur Oikhberg\\
Department of Mathematics\\
University of Illinois\\
Urbana, IL, USA} \email{oikhberg@illinois.edu}

\begin{abstract}
We obtain Lebesgue-type inequalities for the greedy algorithm for arbitrary
complete seminormalized biorthogonal systems
in Banach spaces. The bounds are given only in terms of the upper democracy functions of the basis
and its dual. We also show that these estimates are equivalent to embeddings between the given
Banach space and certain discrete weighted Lorentz spaces. Finally, the asymptotic optimality of
these inequalities is illustrated in various examples of non necessarily quasi-greedy bases.
\end{abstract}


\date{\today}


\maketitle

\section{Introduction and main results}\label{secIntroduc}

Throughout the paper $(\SX,\|\cdot\|)$ is a separable infinite dimensional Banach space over a field
$\mathbb K = \SR$ or $\SC$,  $(\SX^*,\|\cdot\|_*)$ is its dual space, and $\{\be_n, \be_n^*\}_{n=1}^\infty$
 a seminormalized complete biorthogonal system in $\SX$.
To every $x\in\SX$ we associate a formal series $x \sim \sum_{n=1}^\infty \be_n^*(x) \be_n$, so that $\lim_n \be_n^*(x)=0$. It is well-known that greedy algorithms can be considered in this generality \cite{Wo}, which includes in particular the cases when the system $\mathcal B = \{\be_n \}_{n=1}^\infty$ is a Schauder or a Markushevich  basis.

Given $x\in\SX$, the error of $N$-term approximation with respect to $\mathcal B$ is denoted by
$$
  \sigma_N(x) 
	:=
  \inf \Big\{ \big\|x-\sum_{n\in A}c_n\be_n\big\|\mid c_n\in\SK,\;|A|\leq N \Big\}, \quad N=1,2,3, \dots
$$
and the error of the \emph{expansional} $N$-term approximation by
$$
\widetilde \sigma_N(x) := 
\inf \Big\{ \big\|x-\sum_{n\in A}\be^*_n(x)\be_n\big\|\mid |A|\leq N\Big\}, \quad N=1,2,3, \dots\,
$$
A \textbf{greedy set} for $x\in \SX$ of order $N$, written $A\in \mathcal G(x,N)$, is a set of indices $A\subset\SN$ such that $|A|=N$ and
$$
\min_{n\in A} |\be_n^*(x)| \geq \max_{n\notin A} |\be_n^*(x)|\,.
$$
A {\bf greedy operator} of order $N$ is any mapping $G_N:\SX\to\SX$ such that
$$
x\in\SX
\longmapsto  G_Nx =
\sum_{n\in A_x} \be_n^*(x) \be_{n}\,,
$$
with  $A_x \in \mathcal G(x,N)$.  We write $\mathcal G_N$ for the set of all greedy operators of order $N$.

To quantify the performance of greedy operators as $N$-term approximations, one considers, for every $N=1,2,3, \dots$, the smallest numbers $\LN =\LN (\cB,\SX)$ and  $\tLN=\tLN(\cB,\SX)$ such that
\begin{eqnarray} \label{LN}
\Vert x-G_Nx\Vert \,\leq \,\LN \, \sigma_N(x),\quad \forall\;x\in\SX,\quad\forall\;\ G_N\in 
\mathcal{G}_N
\end{eqnarray}
and
\begin{eqnarray} \label{LNE}
\Vert x-G_N(x)\Vert\, \leq\, \tLN\,\tilde{\sigma}_N(x),\; \forall\; x\in\SX, \quad\forall\;G_N\in \mathcal{G}_N.
\end{eqnarray}
As in \cite[Chapter 2]{Tem2015}, we call \eqref{LN} a {\bf Lebesgue-type inequality} for the greedy algorithm, and $\LN $ its associated Lebesgue-type constant.

The question of the performance of $\|x-G_Nx\|$ compared
to $\sigma_N(x)$ was raised by V. N. Temlyakov in the 90s; see \cite{T2,Tem2015} for historical background. Lebesgue-type inequalities were first proved for the trigonometric
and the Haar systems in $L^p$ spaces \cite{Tem98trig,Tem98haar,Tem98mhaar,Wo,Osw2001}. Also, a celebrated result in \cite{KT} established that $\LN =O(1)$ if and only if the system $\cB$ is democratic and unconditional in $\SX$ (also called a \emph{greedy basis}). Nowdays, Lebesgue-type inequalities are reasonably well-understood in the larger class of quasi-greedy bases; see e.g. \cite{TYY2011b,DST,GHO2013,DKO2015}.

For general bases, however, it is a challenging problem to find bounds for $\LN$ which are both, asymptotically optimal and described in terms of reasonable quantities (such as the unconditionality and democracy parameters).
A first approach to this problem was recently given in \cite{BBG2016}. Here we present a different approach, which only depends on the democracy functions of $\cB=\{\be_n\}_{n=1}^\infty$ and $\cB^*=\{\be^*_n\}_{n=1}^\infty$, and allows to cover some cases not considered in \cite{BBG2016}.

\

To describe our results we shall use the following notation. We write $\cY$ for the collection of all $\bfe = \{\e_j\}_{j=1}^\infty\subset\SK$ with $|\varepsilon_j|=1$. For finite sets $A\subset \SN$ we let
$$
{\bf 1}_{\bfe A} := \sum_{j\in A} \varepsilon_j \be_j\,\;\mand  {\bf 1}_{\bfe A}^* := \sum_{j\in A} \varepsilon_j \be_j^*,\quad \bfe\in\cY.
$$
If $\bfe\equiv1$ we just write $\bone_A$ and $\bone^*_A$. We define the \emph{upper (super)-democracy parameters} associated with $\cB$ and $\cB^*$, respectively, by
\Be
D(N): = \sup_{{|A|= N}\atop{\bfe\in\cY}} \|\bone_{\bfe A}\|\,\mand
D^*(N): = \sup_{{|A|= N}\atop{\bfe\in\cY}} \|\bone^*_{\bfe A}\|_*\,.
\label{defDN}
\Ee
For each finite set $A\subset\SN$, we denote by $P_A$ the projection operator \[
P_A(x)= \sum_{n\in A} \be_n^*(x)\be_n,\quad x\in\SX,\] and define the \emph{conditionality constants}
\Be \label{conditional}
K_N = K_N(\mathcal B, \SX):= \sup \big\{\|P_A\| \mid |A| \leq N\big\}\,, \ N=1,2,3, \dots
\Ee
Note that $\cB$ is unconditional if and only if $K_N = O(1)$.
In general, for every given quantity $A_N=A_N(\cB,\SX)$ (such as $K_N,\LN,\tLN$, ...),
we define
\[
A_N^*:=\,A_N(\cB^*,\widehat\SX),
\] where $\widehat \SX := \Cspan{\be_n^*}_{n=1}^\infty$ is considered as a closed subspace of $\SX^*$. In particular, notice that $K_N^* \leq K_N$.

\

To every pair of positive non-decreasing sequences $\{\eta_1(j)\}_{j=1}^\infty$ and $\{\eta_2(j)\}_{j=1}^\infty$, we associate the following numbers
\Bea
S_N (\eta_1, \eta_2) & := & \sum_{j=1}^N \Delta \eta_1(j) \Delta \eta_2(j),\label{SN}\\ \label{TN}
  T_N(\eta_1, \eta_2) & := & \sum_{j=1}^N \frac{ \eta_1(j)}{j} \Delta \eta_2(j),\\
\OT_N(\eta_1,\eta_2) & := & \min\big\{T_N(\eta_1,\eta_2),T_N(\eta_2,\eta_1)\big\}.\label{OTN}\Eea
Here, $\Dt\eta (j) = \eta(j) - \eta(j-1)$, $j=1,2,\dots $ (with the agreement that $\eta(0)=0$). Our main result can then be stated as follows.

\begin{theorem}\label{Th3}
	Let $\{\be_n, \be_n^*\}_{n=1}^\infty$ be seminormalized, complete and biorthogonal in $\SX$.
Let $\OT_N=\OT_N(D,D^*)$ as above. Then the following hold
\Be
  K_N \leq \OT_N\,, \qquad  \LN,\LN^*   \,\leq 1 +3\OT_N\,, \mand
  \tLN,\tLN^* \;\leq 1 + 2\OT_N\,.\label{Th1a}
\Ee
  If, additionally, $D$ (resp. $D^*$) is concave, then $S_N=S_N(D,D^*)\leq\OT_N$ and
	\Be
	 K_N \leq S_N\,, \quad  \LN   \,\leq 1 +3S_N\,, \quad
  \tLN \;\leq 1 + 2S_N\,,\label{Th1b}
\Ee
(respectively, for $K^*_N,\LN^*,\tLN^*$). Finally, these estimates are best possible, in the sense that there exist $\SX$ and $\{\be_n,\be^*_n\}$ for which all the equalities hold.
\end{theorem}

We add a few comments related with Theorem \ref{Th3}. First, the novelty concerns mainly the class of not quasi-greedy and not democratic bases. Indeed, in many such instances we actually obtain $\LN\approx \OT_N$, and in general
we always have $\OT_N\lesssim N$, which was not always the case in \cite{BBG2016}. See $\S\ref{Examples}$ below for various examples, including the trigonometric system in $L^p$.

Secondly, in some special cases, such as for quasi-greedy and democratic $\cB$,
we shall see that $\OT_N\lesssim \ln (N+1)$. This bound does not recover $\tLN\approx1$, but it is best possible for $\LN$, $\LN^*$ and $\tLN^*$, which may all grow to the order $\ln(N+1)$; see e.g. $\S\ref{Example2}$ below.
Another instance occurs when $\{\be_n,\be^*_n\}$ is {\bf bidemocratic} (as in \cite{DKKT}), that is
\Be D(N)D^*(N)\leq c\,N,\quad N=1,2,\ldots\label{bidem}\Ee Then $\OT_N(D,D^*)\leq c\,\ln(N+1)$, and again there exist examples with
$\tLN\approx\tLN^*\approx1$  and $\LN\approx\LN^*\approx\ln (N+1)$; see e.g. the new spaces $KT(p,\infty)$ in $\S\ref{Example5}$ below.

	\
	%
	%

\

To prove Theorem \ref{Th3}, we need to translate the information on $D$ and $D^*$ as embeddings between $\SX$ and a certain family of \emph{discrete weighted Lorentz spaces}. Let $\{s_j^*\}_{j=1}^\infty$ denote the non-increasing rearrangement of a sequence $\{s_n\}_{n=1}^\infty\in c_0$.
Given a non-negative weight $\eta=\{\eta(j)\}_{j=1}^\infty$  we set 
\Be
 \ell^1_\eta=\Big\{\bs\in c_0\mid \|{\bf s}\|_{\ell^1_\eta}:= \sum_{j=1}^\infty \,s_j^*\; \frac{\eta(j)}{j} \, < \infty \,\Big\}.
\label{l1eta}\Ee
We write $\SW$ for the class of all positive increasing weights,
and define, for each $\eta\in\SW$, a new weight
\[\widehat \eta(j)=j \Delta\eta(j),\quad j=1,2,\ldots\]
Below we shall mainly work with the space $\;\ell_{\heta}^1\;$ of all ${\bf s}\in c_0$ with
\Be
 \|{\bf s}\|_{\ell^1_{\heta}}:= \sum_{j=1}^\infty \,s_j^* \,\Delta\eta(j) \, <\, \infty \,.
\label{l1heta}\Ee
 Notice that  $\ell^1_{\eta}$ and $\ell^1_\heta$  are also denoted $d(w,1)$ for  $w_j={\eta(j)}/{j}$ and $w_j=\Delta \eta(j)$, respectively; see e.g. \cite[p. 175]{LZ} or \cite[Example 2.2.3(iv)]{CRS2007}. It is known that  for doubling weights $\eta\in\SW$, both $\ell^1_\eta$ and $\ell^1_\heta$ are quasi-normed spaces; moreover $\ell^1_\eta\subset \ell^1_\heta$, and $\ell_\eta^1 = \ell^1_{\widehat\eta}$ whenever the lower dilation index $i_{\eta}>0$. We shall make a minimum use of these properties in the sequel, but we discuss some of them in $\S\ref{spaces}$ below.

At the other extreme we define the \emph{discrete weighted Marcinkiewicz space} as

\Be \label{marcinkiewicz}
\textit{m}(\eta) =\left\{{\bf s}\in c_0\mid \|{\bf s}\|_{\textit{m}(\eta)}:= \Ts\sup_{k\in \SN}\frac{\eta(k)}k\sum_{j=1}^k s_j^* < \infty  \right\}\,.
\Ee
This is a normed space for every positive $\eta$. We remark that, when $\eta'=\{{j}/{\eta(j)}\}_{j=1}^\infty$, then $\lheta$ and $m(\eta')$ satisfy a duality relation; see \eqref{dualab} below.

\

Finally, we say that a sequence space $\mathbb S$ \textit{embeds into $\SX$ via $\mathcal B$} (with norm $c$), denoted $\mathbb S \stackrel{\cB,\;c}\hookrightarrow \SX\,$, if for every $\textbf{s}=\{s_n\}_{n=1}^\infty\in\mathbb{S}$, there exists a \textbf{unique} $x\in \SX$ such that $\be_n^*(x)=s_n$
and it holds:
\Be \label{embed1}
\|x\| \leq\, c\,\|\bs\|_{\mathbb S}\,=\,c\, \|\{\be_j^*(x)\}_{j=1}^\infty\|_{\mathbb S}\,.
\Ee
Similarly, we say that $\SX$ \textit{embeds into $\mathbb S$ via $\mathcal B$} (with norm $c$), denoted $\SX \stackrel{\cB,\;c}\hookrightarrow \mathbb S\,$, if 
\Be \label{embed2}
\|\{ \be_j^*(x)\}_{j=1}^\infty\|_{\mathbb S}\, \leq \,c \,\| x \|\,,\quad x\in\SX.
\Ee

\medskip

Our two main results concerning embeddings can then be stated as follows.

\medskip

\begin{theorem}\label{Th1}
	Let $ \Been$ be seminormalized and biorthogonal in $\SX$, and $\eta\in \SW$. Then, the following are equivalent:
	
	\sline i) $\;\displaystyle \|{\bf 1}_{\bfe A}\| \leq \eta (|A|)$ for all finite $A\subset \SN$ and all $\bfe \in \cY\,.$
	
\sline	ii) $\;\displaystyle \Vert \sum a_n e_n\Vert_{\SX} \leq \Vert \textbf{a}\Vert_{\ell^1_{\widehat\eta}},\; \mbox{for all} \ \textbf{a}=\{a_n\}\in c_{00}$.
	
\sline Moreover, if $\cB^*$ is total, then each of the above is equivalent to
	
\sline	iii) $\;\displaystyle \ell^1_{\widehat\eta} \stackrel{\cB,\,1}\hookrightarrow \SX$.
\end{theorem}

As noted above, $\;\displaystyle \ell^1_{\widehat\eta}$ is a linear space if and only if
the sequence $\eta$ is doubling.

\begin{theorem}\label{Th2}
	Let $\Been$ be seminormalized biorthogonal and complete in $\SX$, and $\eta$ a positive sequence. Then, the following are equivalent:

	\sline (i) $\;\|{\bf 1}_{\bfe A}^*\|_* \leq  \eta (|A|)$ for all finite $A\subset \SN$ and all $\bfe\in \cY\,.$
	
	\sline (ii) $\;\SX \stackrel{\cB,\,1}\hookrightarrow m(\eta')$, with  $\eta' =\{j/\eta(j)\}_{j=1}^\infty$.
\end{theorem}

The relation between democracy functions and embeddings goes back to early papers in the topic \cite{Wo, GN,GHN}.
A detailed study for quasi-greedy bases was recently given in \cite{AA2016}. Our approach is closer to that in \cite[Proposition 3.6  and Corollary 3.7]{DKO2015}, where bounds for $\LN$ are obtained for general bases under assumptions
of the form $\ell^{q,\infty}\hookrightarrow\SX\hookrightarrow\ell^{p,1}$, where $\ell^{p,r}$ are the classical (unweighted) Lorentz spaces.

\

The outline of the paper is as follows. Section 2 collects preliminaries about bases, weights and discrete Lorentz spaces. The proofs of Theorems \ref{Th1}, \ref{Th2}, and \ref{Th3} are given in sections \ref{embeddings-1}, \ref{embeddings-2}, and \ref{estimates}, respectively. In section \ref{superdemocacydual} we give some estimates for $D^*(N)$, and in
section 7 we present corollaries of Theorem \ref{Th3} in various special cases. Finally, section \ref{Examples} is devoted to examples of optimality, some of them new in the literature.

\

{\small {\bf Acknowledgements}. Second, third and fourth authors supported by grants MTM2014-53009-P, MTM2013-40945-P and MTM2016-76566-P (MINECO, Spain). First, third and fourth authors also partially supported by grant 19368/PI/14 (\textit{Fundaci\'on S\'eneca}, Regi\'on de Murcia, Spain). Last author partially supported by the Simons Foundation travel award 210060.}
\medskip

\section{Preliminaries} \label{Preliminaries}

\subsection{Biorthogonal systems} \label{Bases}
We recall some basic notions; see e.g. \cite{Hajek}.
Let $\SX$ be a separable Banach space, and consider
$\mathcal{B} =\{\be_n\}_{n=1}^\infty \subset \SX$ and
$\mathcal{B^*} = \{\be_n^*\}_{n=1}^\infty \subset \SX^*$. Then the collection  $\Been$ is called

\Benu \item[(a)] a {\bf biorthogonal system} if	 $\;\be^*_n(\be_m)= \delta_{n,m}$ for all  $n,m \in \N$

\item[(b)] \textbf{seminormalized} if there exist $A,B\in(0,\infty)$ such that
$\;A \leq \|\be_n\|, \|\be_n^*\|_* \leq B$ for all $n \in \N$
\Eenu
We additionally say that
\Benu
\item[(c)] $\cB$ is \textbf{complete} in $\SX$ if  $\;\overline{\mbox{span} \{\be_n : n \in \N\}} = \SX.$

\item[(d)] $\cB^*$ is \textbf{total} in $\SX$ if the only $x\in \SX$ such that
$\be_n^*(x)=0$ for all $n \in \N$, is $x=0$. This property is known to be equivalent to
\[
 \overline{\mbox{span} \{\be_n^* : n \in \N\}}^{w^*} = \SX^*.\]
\Eenu
Biorthogonal systems as above are ubiquitous: any separable Banach
space contains, for any $\varepsilon > 0$, a complete and total biorthogonal system
$\{\be_n, \be_n^*\}_{n =1}^\infty$ so that
$1 \leq \|\be_n\|, \|\be_n^*\| \leq 1+\varepsilon$ holds for every $n$
(see \cite[Theorem 1.27]{Hajek}). Specific examples include
Schauder bases and their rearrangements, as well as
the trigonometric system in, for instance, $C(\ST)$ or $L_1(\ST)$.


 In the sequel we shall use the terminology {\bf s-biorthogonal} to denote systems that are seminormalized and biorthogonal. 



\

\subsection{Democracy constants} The definition of \emph{upper (super)-democracy sequence} $D(N)$ was already given in \eqref{defDN}. The following properties are elementary.

\begin{lemma}\label{p0}
The sequence $D(N)$ in \eqref{defDN} is quasi-concave, that is
\[ D(N)\leq D(N+1)\mand \frac{D(N+1)}{N+1}\,\leq\, \frac{D(N)}{N},\quad N=1,2,\ldots\]
\end{lemma}
\Proof
First observe that we can write \Be
D(N)=\sup_{{|A|= N}\atop{\e\in\cY}}\|\bone_{\bfe A}\|=\sup_{{|A|\leq N}\atop{\e\in\cY}}\|\bone_{\bfe A}\|.
\label{DNless}\Ee
Indeed, if $ |A|\leq N$, take any $B\subset\SN$ such that $A\subset B$ and $|B|=N$, and
write $\bone_{\bfe A}= \frac{1}{2}[ (\bone_{\bfe A}+ \bone_{B\setminus A}) + (\bone_{\bfe A}- \bone_{B\setminus A})]$. Then \eqref{DNless} follows from the triangle inequality.

Clearly, \eqref{DNless} implies that  $D(N)$ is non-decreasing.
To see that $D(N)/N$ is non-increasing one can argue as in \cite[p. 581]{DKKT}, that is, for $|A|=N$ write\[
\bone_{\bfe A}=\tfrac1{|A|-1}\sum_{n\in A}\bone_{\bfe (A\setminus\{n\})},
\]
and then use the triangle inequality.
%
\ProofEnd
\

Sometimes we shall also make use of the \emph{lower (super)-democracy sequences}
\Be
d(N):= \inf_{{|A|= N}\atop{\bfe\in\cY}}\|\bone_{\bfe A}\|,\mand \ld(N):=\inf_{{|A|\geq N}\atop{\bfe\in\cY}}\|\bone_{\bfe A}\|.
\label{dN}\Ee
Observe that $\ld(N)$ is non-decreasing, $\ld(N)\leq d(N)$, and if $\cB$ is a Schauder basis, say with constant $\ttM$, then
also $d(N)\leq \,\ttM\;\ld(N)$. In general, however, $\ld(N)$ may be much smaller than $d(N)$.
The corresponding notions for $\mathcal B^*$ will be denoted by $d^*(N)$ and $\ld^*(N)$.

\begin{lemma} \label{p0bis} If $\Been$ is a biorthogonal system in $\SX$, then
\[N\;\le\; \min\big\{\,D(N)\,\ld^*(N)\;, \;D^*(N)\,\ld(N)\;\big\},\quad \forall\;N\in \mathbb N.\]
\end{lemma}
\begin{proof}
 Let $|A| \ge N$ and take any $B\subset A$ with $|B|=N$. Then
$$N=\bone^*_{\bfe A}(\bone_{\bfe B})\le \|\bone_{\bfe B}\|\|\bone^*_{\bfe A}\|_*\le D(N)\|
\bone^*_{\bfe A}\|_*.$$
The result now follows taking infimum over all $|A|\ge N$ and $\bfe\in\cY$. A similar argument gives the other inequality.
\end{proof}

\medskip

Finally, recall from \cite{KT},
that $\mathcal{B}$ is called \emph{superdemocratic} when $\sup_N D(N)/d(N)<\infty$.
In general, we shall quantify superdemocracy with the sequence \Be \label{sdemo}
\mu_N := \sup_{n\leq N}\dfrac{D(n)}{d(n)}.
\Ee

\medskip

\subsection{Abel summation formula} We shall make frequent use of the following elementary identity:
for all finite sequences $\{x_n\}_{n=1}^N$ in $\SX$ and $\{d_n\}_{n=1}^N$ in $\mathbb K$ it holds
\Bea \label{Abel2}
 x_1 d_1 +\sum_{n=2}^N d_n (x_n- x_{n-1})\,=\,\sum_{n=1}^{N-1} (d_n - d_{n+1}) x_n \,+\,x_Nd_N.\Eea

\subsection{Weight classes}\label{weights}

A \emph{weight} is any sequence  $\eta =\{\eta(j)\}_{j=1}^\infty$ of non-negative numbers with $\eta(1)>0$. We use the following notation
\Bi
\item  $\eta>0$ for a \emph{positive} weight, that is,  $\eta(j)>0$ for all $j=1,2,\ldots$
\item  $\SW$ for the set of positive non-decreasing weights, that is,  $0<\eta(1)\leq\eta(2)\leq\ldots$
\item $\Wd$ is the subset of \emph{doubling} weights, that is, $\eta\in\SW$ with $\eta(2j)\leq c\eta(j)$, for some $c\geq1$ and all $j=1,2,\ldots$
\item  $\Wqc$ is the subset of \emph{quasi-concave} weights, that is, $\eta\in\SW$ with
\[
\frac{\eta(j+1)}{j+1}\leq \frac{\eta(j)}j,\quad j=1,2,\ldots
\]
\item  $\Wc$ is the subset of all \emph{concave} weights, that is, $\eta\in\SW$ with \Be\Delta^2\eta(j)= \Delta\eta(j)-\Delta\eta(j-1)\leq 0,\quad \mbox{for $j=2,3,...$},\label{Wc}\Ee
\Ei
Recall from $\S1$ that $\Dt\eta (j) := \eta(j) - \eta(j-1)$, $j=1,2,\dots $, and by convention we always set $\eta(0)=0$.  It is easy to see from the above definitions that \[
\Wc\subset\Wqc\subset \Wd\subset\SW.\]
Also, every $\eta\in\Wqc$ has a \emph{smallest concave majorant} $\coeta\in\Wc$ with $\eta\leq\coeta\leq 2\eta$.
Finally, notice that $D,D^*\in\Wqc$, by Lemma \ref{p0} above.

\medskip

Associated with a weight $\eta$ we consider the following sequences
\Bi
\item \textbf{summing weight:} $\quad\widetilde \eta (N) = \sum_{j=1}^N \frac{\eta(j)}{j}.$
\item \textbf{difference weight:} $\quad\widehat \eta (j) = j\Delta\eta(j)$\quad (if $\eta\in\SW$)
\item \textbf{dual weight:} $\quad \eta'(j) = {j}/{\eta(j)}\;$ \quad (if $\eta>0$).
\Ei
It is elementary to verify the identities:
\Be
\widetilde{{\widehat \eta}}=\eta,\qquad \widehat{{\widetilde \eta}}=\eta, \qquad  (\eta')'=\eta.
\label{theta}\Ee
Moreover, for every $\eta\in\SW$, the following hold
\begin{equation}\label{e2} \eta\in \Wqc \;\Longleftrightarrow \;\widetilde\eta\in \Wc\;\Longleftrightarrow\;\widehat \eta \le \eta\; \Longleftrightarrow \;\eta'\in \Wqc.\end{equation}
Finally observe that, if $\eta\in\SW$, then
\Be
\Ts\teta(N) \leq \eta (N)\sum_{j=1}^N \frac{1}{j} \leq  \eta(N) (1+\ln N) .\label{tetalog}\Ee

\BE\label{Ex2.3}{\color{white}.}

\Benu \item[(i)] If $\eta(j)= [\ln(j+c)]^\ga$, $\ga > 0$, then  $\eta\in\Wc$ (for sufficiently large $c$) and \[
\teta(j)\approx [\ln(j+1)]^{\ga+1}, \qquad \heta (j)\approx [\ln (j+1)]^{\ga-1}.\]

\item[(ii)] If $\eta(j)=j^\al[\ln (j+c)]^\ga$, with $\al\in(0,1)$ and $\ga\in\SR$ (or with $\al=1$ and $\ga\leq0$), then $\eta\in\Wc$ (for sufficiently large $c$) and $\teta\,\approx \,\heta\,\approx\, \eta$.
\Eenu
\EE

\subsection{Regular weights and dilation indices}\label{regular_weights}  Below, we will sometimes be interested in weights $\eta\in\SW$ with
the property
\Be
c_1\eta(N)\;\leq\;\teta(N)
\;\leq\;c_2\,\eta(N),\quad N=1,2,\ldots\label{regular}\Ee
for some $c_1,c_2>0$. We shall call these weights \emph{regular}.
We now give some conditions under which \eqref{regular} holds.
The lower estimate holds trivially with $c_1=1$ when $\eta\in \Wqc$.
More generally, one has the following
\begin{proposition}\label{p1} Let $\eta\in \Wd$ with doubling constant $c$. Then
   \begin{equation}\label{p1-1}\eta(N)\,\le\, \tfrac{c}{\ln 2}\;\widetilde\eta(N), \quad N=1,2,\ldots\end{equation}
  Moreover, $\widetilde \eta\in \Wd$ with doubling constant bounded by $3c/2$.
\end{proposition}
\begin{proof}
If $N=2n+1$,
$$
\teta(N)\geq \sum_{j=n+1}^{2n+1} \frac{\eta(j)}{j}
\geq \eta(n+1)\sum_{j=n+1}^{2n+1} \frac{1}{j} \geq \eta(2n+1)\frac{\ln 2}{c}.
$$
Arguing similarly when $N=2n$ shows (\ref{p1-1}). Finally, the last assertion follows from
\[
\teta(2N)=\sum_{j=1}^{N} \frac{\eta(2j)}{2j} +\sum_{j=1}^{N} \frac{\eta(2j-1)}{2j-1}
\leq\tfrac c2  \sum_{j=1}^{N} \frac{\eta(j)}{j} +c\sum_{j=1}^{N} \frac{\eta(j)}j\,=\, \tfrac {3c}2\,\teta(N).
\]
\end{proof}

The upper bound in \eqref{regular} requires some power growth in $\eta$, as shown in Example \ref{Ex2.3}. This growth is typically quantified with the notion of \emph{dilation index}; see \cite{KPS}.
To each  $\eta>0$, we associate two \emph{dilation sequences}  given by
\Be
\phi_\eta(M) = \inf_{k\geq 1} \frac{\eta(Mk)}{\eta(k)} \mand \Phi_\eta(M) = \sup_{k\geq 1} \frac{\eta(Mk)}{\eta(k)}\,,\quad M=1,2,3, \dots\,
\label{phiM}\Ee
The \emph{lower and upper dilation indices} associated with $\eta$ are defined, respectively, by
\Be \label{Eq-9-7}
i_\eta = \sup_{M>1} \frac{\ln (\phi_\eta(M))}{\ln M} \qquad \mbox{and} \qquad I_\eta = \inf_{M>1} \frac{\ln (\Phi_\eta(M))}{\ln M}\,.
\Ee
For instance, for the weights $\eta$ in Example \ref{Ex2.3} we have $i_\eta=I_\eta=0$ in case (i),
and $i_\eta=I_\eta=\al$ in case (ii). Observe also that
$\phi_{\eta'}(M)=M/\Phi_{\eta}(M)$, so we always have \Be
i_{\eta'}=1-I_\eta.\label{iI1}\Ee

\begin{proposition}\label{p2}  Let $\eta\in \SW$.
Then  $\quad\sup_{N\geq1}\frac{\widetilde \eta(N) }{\eta(N)}<\infty \quad \Longleftrightarrow \quad
i_\eta>0
.$
\end{proposition}
\begin{proof}
 Assume first that $i_\eta > 0$. Then, for some integer $s_0 >1$ we have $\la := \phi_\eta(s_0)>1$. Suppose first that $N=s_0^n$ for some $n=1,2,3, \dots.$ Then,
\Bea \label{L2.1-4}
\widetilde \eta(N) &=& \widetilde \eta(s_0^n) =
\eta(1) + \sum_{k=0}^{n-1} \sum_{j=s_0^k+1}^{s_0^{k+1}} \frac{\eta(j)}{j} \nonumber \\
&\leq& \eta(1) + \sum_{k=0}^{n-1} \eta(s_0^{k+1})\sum_{j=s_0^k+1}^{s_0^{k+1}} \frac{1}{j}  \leq (1+\ln s_0) \sum_{k=0}^{n} \eta(s_0^{k})\,.
\Eea
Now, by definition $\phi_\eta(s_0)\leq \eta(s_0^{k+1})/\eta(s_0^{k})$, and therefore \[
\eta(s_0^k)\leq \la^{-1}\,\eta(s_0^{k+1})\leq...\leq\la^{-(n-k)}\eta(s_0^n),\quad k=0,1,\ldots,n.\] Inserting this expression into \eqref{L2.1-4} we obtain
$$
\teta (N) \leq \frac{1+\ln s_0}{1-\la^{-1}}\;\eta(s_0^n)\,= \,c\, \eta(N)\,.
$$
For arbitrary $N>1$, choose $n\in \SN$ such that $s_0^{n-1} < N \leq s_0^{n}.$ Then,
$$
\widetilde{\eta} (N) = \widetilde{\eta} (s_0^{n-1})+ \sum_{j=s_0^{n-1}+1}^N \frac{\eta(j)}{j} \\
\leq c\eta (s_0^{n-1})+\eta(N)\ln s_0 \lesssim \eta(N).
$$
Conversely, assume that $i_\eta=0$. Then $\phi_\eta(M)=1$ for all $M\ge 2$. In particular, for each $M\ge 2$ there exists
$k_M\in \mathbb N$ with 
$\frac{\eta(Mk_M)}{\eta(k)}\le 2$, $\forall k\ge k_M .$
Therefore
$$\widetilde \eta(Mk_M)\ge \sum_{k=k_M}^{Mk_M}\frac{\eta(k)}{k}\ge \frac{1}{2}\eta(Mk_M)\ln M ,$$
leading to $\sup_N\frac{\widetilde\eta(N)}{\eta(N)}=\infty$.

\end{proof}

\begin{corollary}\label{cor2}
Let $\eta\in\Wqc$. Then $\eta'$ is regular if and only if $I_\eta<1$.
\end{corollary}
\Proof
First, $\eta\in\SW$ already implies $\eta'(N)=N/\eta(N)\leq\sum_{j=1}^N1/\eta(j)=\widetilde{(\eta')}(N)$.
Next $\eta \in \Wqc$ implies $\eta'\in\SW$, and by Proposition \ref{p2}, the converse inequality $\widetilde{(\eta')}\lesssim \eta'$ is equivalent to $i_{\eta'}>0$,
and the result follows from the identity in \eqref{iI1}.\ProofEnd

\medskip

\subsection{Weighted Lorentz spaces.} \label{spaces}

We recall a few  basic properties of the class of discrete weighted Lorentz spaces.
Although not necessary for the proofs of Theorems \ref{Th3}, \ref{Th1} and \ref{Th2}, this subsection
clarifies the role of the different conditions we impose on the theorems.

For a non-negative weight $\eta$ and $0<r\leq\infty$ we let
\Be \label{discrete}
\ell_\eta^r =\left\{{\bf s}=\{s_n\}_{n=1}^\infty\in c_0: \|{\bf s}\|_{\ell_\eta^r}:= \Big(\sum_{j=1}^\infty \big[s_j^* \eta(j)\big]^r \tfrac{1}{j}\Big)^{1/r} < \infty  \right\}\,
\Ee
(with the obvious modification if $r=\infty$). In the literature  $\ell^r_\eta$ is sometimes denoted
$d(r,w)$ with $w_j=\frac{\eta(j)^r}{j}$, and the weight $w$ is required to decrease to 0 and $\sum_{j=1}^\infty w_j=\infty$; see e.g. \cite[p. 175]{LZ} or references in \cite[p. 28]{CRS2007}. We will be dealing only with the case $r=1$ but we shall consider more general weights, namely $w=\{\eta(j)/j\}$ and  $w=\Dt\eta$, for $\eta\in \SW$.

It is well-known that $d(1,w)$ are quasi-normed spaces if and only if $W(N)=\sum_{j=1}^N w_j$ satisfies a doubling condition (see \cite[Theorem 2.2.16]{CRS2007}).
 Hence  $\tilde\eta\in \Wd$ implies that $\ell^1_\eta$ is quasi-normed, and
 $\eta\in \Wd$ implies that both $\ell^1_{\widehat\eta}$ and $\ell^1_{\eta}$ are quasi-normed (by \eqref{theta} and Proposition \ref{p1}).

Clearly if $\eta\in  \Wqc$ then $\widehat \eta\le \eta$ and therefore $\ell^1_{\eta}\hookrightarrow \ell^1_{\widehat\eta}$. Below we show that this is the case also for $\eta\in \Wd$.
The following basic lemma will be used often.

\begin{lemma}\label{A2}
If $\nu,\xi$ are non-negative sequences, the following holds
	$$\tilde{\nu}\leq \tilde{\xi}\quad\Longleftrightarrow \quad\sum_{j=1}^\infty a_j^*\dfrac{\nu(j)}{j}\leq \sum_{j=1}^\infty a_j^*\dfrac{\xi(j)}{j},\quad\forall\;\mbox{non-increasing}\;a^*_j.$$
	In particular, $\tilde{\nu}\leq \tilde{\xi}$ if and only if $\ell^1_\xi\hookrightarrow \ell^1_\nu$ with embedding of norm $1$.
\end{lemma}

\begin{proof}
	Suppose that $\tilde{\nu}\leq \tilde{\xi}$. Then, using the Abel summation formula in \eqref{Abel2},
	\begin{eqnarray*}
		\sum_{j=1}^N a_j^*\dfrac{\nu(j)}{j} &=& a^*_1\nu(1)+\sum_{j=2}^N a_j^*(\tilde{\nu}(j)-\tilde{\nu}(j-1)) = \sum_{j=1}^{N-1}(a_j^* - a_{j+1}^*)\tilde{\nu}(j) + a_N^*\tilde{\nu}(N)\\
		&\leq& \sum_{j=1}^{N-1}(a_j^* - a_{j+1}^*)\tilde{\xi}(j)+a_N^*\tilde{\xi}(N) = \sum_{j=1}^N a_j^*\dfrac{\xi(j)}{j}.
	\end{eqnarray*}	
	Now, let $N\rightarrow\infty$ and we obtain the result.	To show the other implication, we only have to take $a_j^* = 1$ if $1\leq j\leq N$ and $a_j^*=0$ in other case.
\end{proof}

\begin{corollary}\label{p1bis} 
\sline (i) If $\eta\in \Wd$, then 
$\ell^1_\eta\hookrightarrow \ell^1_{\widehat\eta}$.

\sline (ii) If $\eta\in \SW$, then $\quad \ell^1_{\widehat\eta}
\hookrightarrow \ell^1_\eta \quad \Longleftrightarrow \quad
i_\eta>0\,.$
\sline (iii) If $\eta\in \Wd$, then $\quad \ell^1_\eta=\ell^1_{\widehat\eta}\quad \Longleftrightarrow \quad
i_\eta>0\,.$
\end{corollary}

\Proof Combine Lemma \ref{A2} with \eqref{theta} and Propositions \ref{p1} and \ref{p2}.
\ProofEnd

\begin{corollary} \label{c0} $\ell^1_{\eta}=c_0$ if and only if $\;\teta\;$ is bounded. In particular,
$\ell^1_{\heta}=c_0$ if and only if $\eta\in\SW$ is bounded.\end{corollary}
\Proof The inclusion $\ell^1_\eta\hookrightarrow c_0$ is always true. For the converse, write $c_0=\ell^1_\xi$ with $\xi=\{1,0,0,\ldots\}$, and use Lemma \ref{A2}.\ProofEnd




We now turn to the \emph{discrete weighted  Marcinkiewicz space} defined in (\ref{marcinkiewicz}), which we compare with the Lorentz space $\ell^\infty_\eta$ in \eqref{discrete}. Observe that, for $\bs=\{s_n\}\in c_0$,
 \begin{equation}\label{ms}
 \|\bs\|_{m(\eta)}=\sup_{{|A|=N}\atop{N\in\SN}}\tfrac{\eta(N)}N\sum_{n\in A} |s_n|.
 \end{equation}
So $m(\eta)$ is a normed space for any non-negative weight $\eta$.
It is also easy to see that, $m(\eta)=c_0$ if and only if $\eta$ is bounded.

 \begin{lemma}\label{L2.2}
\sline (i) $m(\eta)\hookrightarrow \ell^\infty_{\eta}$, with embedding norm 1.

\sline (ii) If $\eta\in \SW$, then  $\;\ell^\infty_\eta\stackrel{c}\hookrightarrow m(\eta)\; $ if and only if $\; \widetilde{(\eta')}\le c \eta'\;$, that is
\Be\sum_{j=1}^N\frac{1}{\eta(j)}\le \,\frac{c\,N}{\eta(N)},\quad N=1,2,\ldots\label{1overeta}\Ee

\sline (iii) If $\eta\in \Wqc$, then $\quad m(\eta) = \ell^\infty_{\eta}  \quad\Longleftrightarrow\quad I_\eta<1$.
\end{lemma}

\begin{proof}
(i) This follows easily from $s^*_N\leq \frac1N\sum_{j=1}^N s_j^*$.
\sline (ii) Assume that $\ell^\infty_\eta\stackrel{c}\hookrightarrow m(\eta)$. Then picking ${\bf s}=\{{1}/{\eta(j)}\}\in \ell^\infty_{\eta}$ we obtain
$$\|\bs\|_{m(\eta)}=\sup_N\frac{\eta(N)}{N}\sum_{j=1}^N \frac{1}{\eta(j)}\le c\|\bs\|_{\ell^\infty_\eta}=c,$$ obtaining the condition \eqref{1overeta}.
Conversely, \eqref{1overeta} and the inequality $s_j^*\le \frac{\|\bs\|_{\ell^\infty_\eta}}{\eta(j)}$
easily lead to $\|\bs\|_{m(\eta)}\leq c\|\bs\|_{\ell^\infty_\eta}$.

\sline (iii)  This follows from (i), (ii) and Corollary \ref{cor2}.
\end{proof}


We conclude with a duality result which is known in the literature; see \cite[$\S2.4$]{CRS2007}. We present an elementary proof.

\begin{theorem} \label{dual} If $\eta\in \Wd$ and $\Ds\inf_{N\in\SN}\tfrac{\eta(N)}{N}=0$, then $(\ell^1_{\widehat\eta})^*=m(\eta')$ isometrically.
\end{theorem}
\begin{proof} Let ${\bf a}\in \ell^1_{\widehat\eta}$ and ${\bf b}\in m(\eta') $.
We may apply Lemma \ref{A2} with $\nu(j)=jb^*_j$ and $\xi(j)=\|{\bf b}\|_{m(\eta')}\,\heta(j)$, since $\widetilde{\nu}(n)\leq \,\widetilde{\xi}(n)$,  and conclude that
\begin{eqnarray}\sum_{j=1}^\infty |a_jb_j|&\le& \sum_{j=1}^\infty a_j^*b_j^*\,\leq\, \|{\bf b}\|_{m(\eta')}\,\sum_{j=1}^\infty a_j^*\Delta\eta(j)\;= \; \|{\bf b}\|_{m(\eta')}\|{\bf a}\|_{\ell^1_{\widehat\eta}}. \label{dualab}
\end{eqnarray}
This shows that $m(\eta')\subseteq (\ell^1_{\widehat\eta})^*$.
Conversely  let $\Phi\in (\ell^1_{\widehat\eta})^*$ and denote $b_n=\Phi({\bf e}_n)$ with $\{{\bf e}_n\}$ the standard basis in $c_{00}$. If $\e_n=\sign(b_n)$, then for each  $|A|=N$ we have
\Be\sum_{n\in A} |b_n|=|\sum_{n\in A} \bar\e_n b_n|\le \|\Phi\| \|1_{\bar\bfe A}\|_{\ell^1_{\widehat \eta}}=\|\Phi\|\; \eta(N).\label{bjA}\Ee
We claim that ${\bf b}=\{b_n\}_{n=1}^\infty\in c_0$. Indeed, if not, there would be a $\dt>0$ and a subsequence $|b_{n_j}|\geq\dt$, and \eqref{bjA} gives $\dt N \leq \sum_{j=1}^N|b_{n_j}|\le \|\Phi\|\eta(N)$, which contradicts the property $\inf_N\eta(N)/N=0$. Finally, \eqref{bjA} implies that $\sum_{j=1}^N b_j^*\le \|\Phi\|\eta(N)$, and therefore $\|{\bf b}\|_{m(\eta')}\le \|\Phi\|$. This completes the proof of the theorem.
\end{proof}

\subsection{Properties of $T_N(\eta_1,\eta_2)$}
We show elementary relations for 
\[
S_N(\eta_1,\eta_2),\quad T_N(\eta_1,\eta_2),\mand \OT_N(\eta_1,\eta_2),\]
defined in \eqref{SN}-\eqref{OTN}, and also for the quantity
\Be \label{UN}
U_N(\eta_1, \eta_2) :=
\, \sum_{j=1}^N \frac{\eta_1(j)\,\eta_2(j)}{j^2}\, ,\quad N=1,2,\ldots
\Ee

	%
	%
%
%
	%
%
%
	%
%
	
\begin{lemma} \label{Cor2.13}If $\eta_1,\eta_2\in \Wqc$ then
\Be S_N(\eta_1,\eta_2)\leq \OT_N(\eta_1, \eta_2)\le \max\{T_N(\eta_1, \eta_2),T_N(\eta_2,\eta_1)\}\le   U_N(\eta_1,\eta_2).\label{STTU}\Ee
Moreover if we assume that $i_{\eta_1} i_{\eta_2}>0$ then \Be\label{OTUN} \OT_N(\eta_1, \eta_2)\approx U_N(\eta_1,\eta_2).\Ee Finally, if $\eta_1\in\Wc$ and $i_{\eta_2}>0$ (or $\eta_2\in\Wc$ and $i_{\eta_1}>0$),  then
\Be \label{SOT}S_N(\eta_1,\eta_2)\approx \OT_N(\eta_1, \eta_2).\Ee
\end{lemma}
\Proof
The assertion \eqref{STTU} follows easily using that $\Delta \eta(j)\le {\eta(j)}/{j}$ when $\eta\in\Wqc$. If $i_{\eta_2}>0$, we can apply Corollary \ref{p1bis}.ii  to obtain $$U_N(\eta_1,\eta_2)= \|\{\eta_1(j)/j\}_{j=1}^N\|_{\ell^1_{\eta_2}}\le c\,\|\{\eta_1(j)/j\}_{j=1}^N\|_{\ell^1_{\heta_2}}=c\, T_N(\eta_1,\eta_2). $$
If $i_{\eta_1}>0$, a symmetric argument gives $U_N(\eta_1,\eta_2)\leq cT_N(\eta_2,\eta_1)$, and hence \eqref{OTUN}.	
Finally, if $\eta_1\in\Wc$ and $i_{\eta_2}>0$, then Corollary \ref{p1bis}.ii  gives
\Be T_N(\eta_2,\eta_1)=  \|\{\Delta\eta_1(j)\}_{j=1}^N\|_{\ell^1_{\eta_2}}\le c\,\|\{\Delta\eta_1(j)\}_{j=1}^N\|_{\ell^1_{\heta_2}}=\,c\,S_N(\eta_1,\eta_2),\label{TNSN}\Ee
which together with \eqref{STTU} gives \eqref{SOT}. A similar reasoning works interchanging $\eta_1$ and $\eta_2$.
\ProofEnd	
\BE
If $\eta_1(j) =j$ and $\eta_2(j)=1$ for all $j=1,2,3, \dots.$ Then, $i_{\eta_1}=1$, $i_{\eta_2}=0$ and  $$ S_N(\eta_1,\eta_2)= T_N(\eta_1, \eta_2)=1,   \quad T_N(\eta_2,\eta_1)= U_N(\eta_1, \eta_2)  \approx \ln(N+1)\,.$$
Hence,  \eqref{OTUN} may not hold if $i_{\eta_2}=0$.
\EE

%

\begin{lemma}\label{etaxi}
Let $\eta_1,\eta_2,\xi_2$ be non-negative sequences with 
$\eta_2\leq\xi_2$.
	\sline (i)  If  $\eta_1\in\Wqc$, then $T_N(\eta_1,\eta_2)\leq T_N(\eta_1,\xi_2)$.
	\sline (ii) If $\eta_1\in\Wc$, then $S_N(\eta_1,\eta_2)\leq S_N(\eta_1,\xi_2)$.
	\end{lemma}
\Proof (i) is elementary using Abel's formula \eqref{Abel2}:
\[
T_N(\eta_1,\eta_2)=\sum_{j=1}^{N-1}\Big[\frac{\eta_1(j)}j-\frac{\eta_1(j+1)}{j+1}\Big]\,\eta_2(j)+\frac{\eta_1(N)}N\eta_2(N)\leq T_N(\eta_1,\xi_2),
\]
since $\eta_1\in\Wqc$ and $\eta_2\leq \xi_2$. The proof of (ii) is similar.
\ProofEnd

\section{Embeddings of discrete spaces into $\SX$} \label{embeddings-1}

\subsection{Proof of Theorem \ref{Th1}}
	The implication $ii) \Rightarrow i)$ is clear since
	$$
	\|\bone_{\bfe A}\| \leq \big\|\{\e_j\}_{j\in A}\big\|_{\ell^1_\heta}=
	\sum_{j=1}^{|A|} \Dt\eta(j) = \eta(|A|)\,.
	$$
	
We now show that $i)\Rightarrow ii)$. Let $\textbf{a}\in c_{00}$ and $N=|\supp\textbf{a}|$. Write $a_j^* = \vert a_{\pi(j)}\vert$, where $\pi:\lbrace 1,...,N\rbrace\rightarrow \supp\textbf{a}$ is a greedy bijection, that is $\vert a_{\pi(j)}\vert \geq \vert a_{\pi(j+1)}\vert$, $j=1,2,...$ Let also $\varepsilon_j = \sign(a_{\pi(j)})$.
If we define \Be
	\displaystyle S_J = \sum_{j=1}^J \varepsilon_j \be_{\pi(j)}\label{SJ}\Ee
(and $S_0=0$), then using Abel summation formula (\ref{Abel2})  we can write
	$$\sum_{n\in \supp\textbf{a}} a_n \be_n = \sum_{j=1}^Na^*_j\e_j\be_{\pi(j)}=\sum_{j=1}^N a_j^* (S_j -S_{j-1}) = \sum_{j=1}^{N-1} (a_j^* - a_{j+1}^*)S_j + a_N^* S_N.$$
	Then, by assumption $i),$
	\begin{eqnarray}
		\Big\Vert \sum_{n\in\supp\textbf{a}}a_n \be_n\Big\Vert &\leq& \sum_{j=1}^{N-1} (a_j^* - a_{j+1}^*)\Vert S_j\Vert + a_N^*\Vert S_N\Vert \leq \sum_{j=1}^{N-1} (a_j^* - a_{j+1}^*)\eta(j) + a_N^* \eta(N)\nonumber\\
		&=& a_1^* \eta(1) + \sum_{j=2}^{N} a_j^* (\eta(j)-\eta(j-1)) = \Vert \textbf{a}\Vert_{\ell^1_{\widehat\eta}},\label{a_suppN}
	\end{eqnarray}
which is the desired result.
	
	The implication $iii)\Rightarrow ii)$ is immediate, so  it remains to prove $i) \Rightarrow iii)$ under the assumption that the system is total. Let $\textbf{a} \in \ell^1_{\widehat\eta}$,
	which we shall assume with infinite support (otherwise we may use \eqref{a_suppN}).
	As before, write $a_j^* = \vert a_{\pi(j)}\vert$ where $\pi : \SN \rightarrow \supp\textbf{a}$ is a greedy bijection, and $\e_j = \sign(a_{\pi(j)})$.
Letting $S_J$ be as in \eqref{SJ}, we have
\begin{eqnarray*}
		\sum_{j=1}^J (a_j^* - a_{j+1}^*)\Vert S_j\Vert &\leq& \sum_{j=1}^J (a_j^* - a_{j+1}^*)\eta(j)\\
\mbox{{\footnotesize (by \eqref{Abel2})}}		&=& a_1^* \eta(1) + \sum_{j=2}^{J+1} a_j^*[\eta(j)-\eta(j-1)] -
a_{J+1}^*\eta(J+1)\\
		&\leq & \sum_{j=1}^\infty a_j^*\Delta\eta(j) = \Vert \textbf{a}\Vert_{\ell_{\heta}^1}<\infty.
	\end{eqnarray*}	
Therefore, the	series  $\sum_{j=1}^\infty (a_j^* - a_{j+1}^*)S_j $ converges to some $x\in \mathbb{X}$ and $\|x\|\le \|\textbf{a} \|_{\ell^1_{\widehat\eta}}$. It only remains to show that \Be
\be^*_n(x)=a_n,\quad \forall\; n\in\SN.\label{enx}\Ee If $n\not\in \supp\textbf{a}$ then
$\be_n^*(S_j)=0$ for all $j$, and thus $\be_n^*(x)=0$.
Let then $n\in \supp\textbf{a}$, and write $n = \pi(j_n)$, so that
	\begin{eqnarray*}
		\be_n^*(x) &=& \lim_{J\rightarrow\infty} \be_n^*\left( \sum_{j=1}^J (a_j^* - a_{j+1}^*)S_j\right) = \lim_{J\rightarrow\infty} \sum_{j=j_n}^J (a_j^* - a_{j+1}^*)\varepsilon_{j_n}\\
		&=& \lim_{J\rightarrow\infty} (a_{j_n}^* - a_{J+1}^*)\varepsilon_{j_n} = a_{j_n}^*\varepsilon_{j_n}=a_{\pi(j_n)}=a_n,
	\end{eqnarray*}
	where we have used that $\textbf{a}\in c_0$.
	Finally, there is a unique element $x$ with the property \eqref{enx} by the totality of the system $\cB^*$.
This shows that $\ell^1_\heta \stackrel{\cB,1}\hookrightarrow \SX$, and completes the proof of the theorem.
\ProofEnd


\BR
	The statement of Theorem \ref{Th1} resembles a well known property of the classical Lorentz spaces $L^{p,1}$. Namely, if $\|\cdot\|$ is an order preserving norm defined on the set $\mathcal{S}$ of all simple functions of a measure space $(\Omega,\Sigma, d\mu)$, then the inequality $\|\chi_E\| \leq \mu(E)^{1/p}$ for all $E\in\Sigma$, implies that $\|f\| \leq \|f\|_{L^{p,1}(\mu)}$ for all $f\in\mathcal{S}$; see \cite[Thm V.3.11]{SW1971}.
\ER

\BR In the special setting of quasi-greedy bases, a result similar to Theorem \ref{Th1} was proved earlier by the fourth author in \cite[Lemma 2.1]{H2011}. More precisely, if $\cB$ is quasi-greedy in $\SX$ and $\eta\in\Wd$ is such that $\|\bone_A\| \leq \eta(|A|),$ then $\ell_\eta^1\hookrightarrow \SX$ via $\cB$. Theorem \ref{Th1}  actually shows that one can choose a better space, since $\ell_{\eta}^1 \subset \ell^1_{\widehat\eta}$. See also \cite[Theorem 3.1]{AA2016}.
\ER

\

\section{Embeddings of $\SX$ into discrete spaces} \label{embeddings-2}

%

\subsection{Proof of Theorem \ref{Th2}}
	$(i) \Rightarrow (ii).$ For $x\in\SX$, write $a_j^*(x)=\vert \be_{\pi(j)}^*(x)\vert$, where $\pi$ is a greedy permutation onto $\supp x$, that is, $\vert \be_{\pi(j)}^*(x)\vert \geq \vert \be_{\pi(j+1)}^*(x)\vert$, $j=1,2,...$. We also let $\varepsilon_j = \sign(\be_{\pi(j)}^*(x))$, $j=1,2...$ Then
	\Beas
	\frac{\eta'(N)}{N}\sum_{j=1}^N a_j^*(x) &=&
	\frac{1}{\eta(N)}\sum_{j=1}^N a_j^*(x) =\frac{1}{\eta(N)}\Big(\sum_{j=1}^N \bar\e_j \be_{\pi(j)}^*\Big)(x) \\
	&\leq & \frac{1}{\eta(N)}\Vert \sum_{j=1}^N \bar\e_j \be_{\pi(j)}^*\Vert_*\Vert x\Vert \leq \Vert x\Vert.
	\Eeas
	
	$(ii) \Rightarrow (i).$  Let $A \subset \SN$ be a finite set  and $\bfe \in \cY.$ Then,
\Be
	\|\bone^*_{\bfe A}\|_* = \sup_{\|x\|=1}  |\bone^*_{\bfe A}( x)|  = \sup_{\|x\|=1} |\sum_{j\in A} \e_j\,\be_j^*(x)|\,.
\label{aux_thm2}\Ee
Now, given $x\in\SX$, and denoting $\lbrace a_j^*(x)\rbrace_j$ as in the proof of the previous implication,
we have \[
 \Big|\sum_{j\in A} \e_j\,\be_j^*(x)\Big| \leq \sum_{j\in A} |\be_j^*(x)| \leq \sum_{j=1}^{|A|} a_j^*(x)
\leq \eta(|A|)\|x\|,\]
with the last inequality due to the assumption (ii). Inserting this estimate into \eqref{aux_thm2} gives
the desired expression (i).
\ProofEnd

\BR
In the setting of quasi-greedy bases,
a different version of Theorem \ref{Th2} involving lower democracy function $h_\ell$ was
proved in \cite[Lemma 2.2]{H2011}. Namely, $\SX\hookrightarrow \ell^
\infty_{h_\ell}$; see also \cite[Theorem 3.1]{AA2016}.
Such embedding, however, cannot hold for general bases. For instance, consider the space $\SX$ of all sequences $\textbf{a}=\{a_n\}_{n=1}^\infty\in c_0$ with
$$
\|\textbf{a}\| := \sup_{M\geq 1} \Big|\sum_{n=1}^M a_n\Big| < \infty\,,
$$
with the standard canonical basis $\{\be_n\}$.
Then, $h_\ell(N) = \inf_{|A|=N}\|\bone_A\|=N$. However, the embedding $\SX \hookrightarrow \ell_{h_\ell}^\infty$ cannot hold since $\ba=\{ (-1)^n\}_{n=1}^N$ belongs to $\SX$ with $\|\ba\|=1$, but $\sup_n n a_n^* = N\to \infty.$
\ER

\section{Proof of Theorem \ref{Th3}} \label{estimates}

The results we prove here are slightly stronger than those announced in Theorem \ref{Th3}.
Throughout this section, the sequences $\eta_1,\eta_2\in {\mathbb{W}}$ are such that
\Be
 (1)\quad\|\bone_{\bfe A}\| \leq \eta_1(N) \qquad \mbox{and} \qquad (2) \quad \|\bone_{\bfe A}^*\|_* \leq \eta_2(N) ,
\quad\forall\;|A|=N,\;\forall\;\bfe\in\cY.
\label{1A}\Ee
As noted above, these inequalities are satisfied for $\eta_1 = D$ and $\eta_2 = D^*$.

\subsection{Estimates for $K_N$}
Instead of estimating $K_N$, we work with the larger quantity
\[
\bKN=\sup_{{|A|\leq N}\atop{\bfe\in\cY}}\|P_{\bfe A}\|,
\]
where $P_{\bfe A}x=\sum_{n\in A}\e_n\be^*_n(x)\be_n$. 

\begin{lemma}\label{est for K_N}
Suppose the sequences $\eta_1,\eta_2 \in {\mathbb{W}}$ satisfy \eqref{1A}. Then:
\sline (i)
If $\eta_1 \in \Wqc$, then $\bKN \leq \OT_N(\eta_1, \eta_2)$.
 \sline (ii)
If $\eta_1 \in \Wc$, then $\bKN \leq S_N(\eta_1, \eta_2)$.

\end{lemma}


\begin{proof}
 Given any $x\in\SX$, we denote by $\{a_j^*(x)\}$ the decreasing rearrangement of $\{\be^*_n(x)\}$,
that is, $a_j^*(x)=\vert \be_{\pi(j)}^*(x)\vert$, where $\pi$ is a greedy bijection onto $\supp x$;
see the proof of the Theorem \ref{Th2}.
If $|A|\leq N$ and $\bfe\in\cY$, then part (1) of \eqref{1A} and the implication $i) \Rightarrow ii)$
of Theorem \ref{Th1} imply
\Be
\|P_{\bfe A}x\| \leq  \sum_{j=1}^{|A|} a_j^*(P_{\bfe A}x)\Delta\eta_1(j)\,\le  \sum_{j=1}^N a_j^*(x)\Dt\eta_1(j)\,=:A_N(x),
\label{first}\Ee
the last inequality due to $a_j^*(P_{\bfe A}x)=a^*_j(P_Ax) \leq a_j^*(x)$ (and $\eta_1\in\SW$).

We start by proving $(ii)$. Denoting $S_J(x) =\sum_{j=1}^J a_j^*(x)$, and using the Abel summation formula \eqref{Abel2}
\Bea \label{Eq-5-1}
A_N(x) &=&   S_1(x)\,\eta_1(1)+\sum_{j=2}^{N} [S_j(x) - S_{j-1}(x)]\Dt \eta_1(j)  \nonumber \\
&=& \sum_{j=1}^{N-1} \Big[\Dt\eta_1(j) - \Dt\eta_1(j+1)\Big] S_j(x) + \Dt\eta_1(N) S_N(x) \,.
\Eea
Now, the inequality (2) in \eqref{1A} and $i) \Rightarrow ii)$ of Theorem \ref{Th2} imply that
\Be\frac{1}{\eta_2(j)} S_j(x) = \frac{\eta_2'(j)}{j} \sum_{k=1}^j a_k^*(x)\leq \|x\|\,,\quad j =1,2, \dots
\label{insert}\Ee
Since $\eta_1\in\Wc$, we may insert in  (\ref{Eq-5-1}) the inequalities for $S_j(x)$ in (\ref{insert}), and then another use of (\ref{Abel2}) gives,
\Beas \label{PAco2}
A_N(x) &\leq& \Big[\sum_{j=1}^{N-1} [\Delta \eta_1(j) - \Delta \eta_1(j+1)] \eta_2(j) + \Delta \eta_1(N) \eta_2(N) \Big] \|x\|  \nonumber \\
&=& \sum_{j=1}^{N}  \Delta \eta_1(j) \Delta \eta_2(j)\;\|x\| = S_N(\eta_1,\eta_2)\; \|x\|\,.
\Eeas
Plug this into \eqref{first} to obtain the desired estimate for $\bKN$.

To prove $(i)$ assume $\eta_1 \in \Wqc$. Then $\eta_1 \leq \widetilde {\eta_1}$, so that \eqref{1A} holds with $\eta_1$ replaced by $\widetilde {\eta_1}.$
Since $\widetilde{\eta_1}\in \Wc$ (see (\ref{e2})), by part $(ii)$ of this Lemma (just proved)
$$
\bKN\leq S_N(\widetilde{\eta_1}, \eta_2) = \sum_{j=1}^N \frac{\eta_1(j)}{j} \Delta \eta_2(j) = T_N(\eta_1, \eta_2)\,.
$$
On the other hand, observe that in \eqref{first} we could also argue  as follows
\Beas
A_N(x) & =&   \sum_{j=1}^{N} a_j^*(x)\Delta\eta_1(j)\,\le
\sup_{k\in\SN}\big[\tfrac{k}{\eta_2(k)}a^*_k(x)\big]\,\sum_{j=1}^N \frac{ \eta_2(j)}{j}\,\Dt\eta_1(j)\\
& \leq & \|\{a^*_j(x)\}\|_{m(\eta'_2)}\,T_N(\eta_2,\eta_1)\;\leq\;\|x\|\,T_N(\eta_2,\eta_1),
\Eeas
the last inequality due to (2) in (\ref{1A}) and $(i)\Rightarrow(ii)$ in Theorem \ref{Th2}. Thus, we have shown that \eqref{1A} implies
$$
\bKN\leq \min\big\{T_N(\eta_1,\eta_2),\,T_N(\eta_2,\eta_1)\big\}\,=:\,\OT_N(\eta_1,\eta_2) .
$$
\end{proof}

\subsection{Estimates for $\LN$}
\begin{lemma}\label{est for L_N}
Suppose the sequences $\eta_1,\eta_2 \in {\mathbb{W}}$ satisfy \eqref{1A}. Then:
\sline (i)
If $\eta_1 \in \Wqc$, then $\LN \leq 1 + 3\OT_N(\eta_1, \eta_2)$.
\sline (ii)
If $\eta_1 \in \Wc$, then $\LN \leq 1 + 3 S_N(\eta_1, \eta_2)$.
\end{lemma}

\begin{proof}
We follow the standard approach in \cite{KT}. Let $x\in \SX$ and  write $G_Nx=P_\Gamma x$ for some $\Gamma \in \mathcal G(x,N)$. Take any $z=\sum_{n\in B}c_n\be_n$ with $|B|\leq N$. Then,
\Bea \label{Eq-5-3}
\|x - G_Nx\| &=& \|x-P_{B\cup \Gamma}(x)+P_{B\setminus \Gamma}(x) \| \nonumber \\
& \leq &  \|P_{(B\cup \Gamma)^c}(x)\|+\|P_{B\setminus \Gamma}(x)\| =: I + II\,.
\Eea
For the first term we use that $P_{(B\cup \Gamma)^c}(x)=P_{(B\cup \Gamma)^c}(x-z)$, and therefore
\Bea \label{Eq-5-4}
I =\|(I-P_{B\cup \Gamma})(x-z)\|& \leq &  \|x-z\| + \|P_B(x-z)\| + \|P_{\Gamma\setminus B}(x-z)\| \nonumber \\
&\leq&  (1+2K_N)\|x-z\|\,.
\Eea
To estimate $II$ we proceed as follows. First, using \eqref{1A} and $ i) \Rightarrow ii)$ in Theorem \ref{Th1},
$$
II = \|P_{B\setminus \Gamma}(x)\| \leq 
\sum_{j=1}^{|B\setminus \Gamma|} a_j^*(P_{B\setminus \Gamma}(x)) \Delta \eta_1(j)\,\,\leq\,
\sum_{j=1}^{|\Gamma\setminus B|} a_j^*(P_{\Gamma\setminus B}(x)) \Delta \eta_1(j)\,,
$$
where in the last step  we have used that $\Gamma$ is a greedy set for $x$ and
$|B\setminus \Gamma| \leq |\Gamma \setminus B|$.
Now, $P_{\Gamma\setminus B}(x)=P_{\Gamma\setminus B}(x-z)$, and we may
use that $a_j^*(P_{\Gamma\setminus B}(x-z)) \leq a_j^*(x-z)$ to conclude
$$
II = \|P_{B\setminus \Gamma}(x)\| \leq 
\sum_{j=1}^{|\Gamma\setminus B|} a_j^*(x-z) \Delta \eta_1(j) \,.
$$
The right hand side resembles that of \eqref{first}, with
$A_N(x)$ replaced by $A_{|\Gamma\setminus B|}(x-z)$.
We estimate $A_{|\Gamma\setminus B|}(x-z)$ as in Lemma \ref{est for K_N}.
For $\eta_1\in\Wc$ (case $(ii)$), we obtain
\Bea \label{Eq-5-5}
II = \|P_{B\setminus \Gamma}(x)\|\, \leq  \, S_N(\eta_1, \eta_2)\,\|x-z\|.
\Eea
Thus, combining (\ref{Eq-5-3}), (\ref{Eq-5-4}), and (\ref{Eq-5-5}), together with Lemma \ref{est for K_N} ($ii$),
we are led to
\Beas
\| x - G_Nx\| & \leq & (1+2K_{N}+ S_N(\eta_1,\eta_2))\,\|x-z\|\\
& \le &  (1+ 3\, S_N(\eta_1,\eta_2))\|x-z\|\,.
\Eeas
Taking the infimum over all such $z$ we finally obtain
$\LN\leq 1+ 3\, S_N(\eta_1,\eta_2)$.

For $\eta_1\in\Wqc$ (case $(i)$), we modify the preceding argument
(as we did in the proof of Lemma \ref{est for K_N} $(i)$) to obtain $\LN\leq 1+ 3\,\OT(\eta_1,\eta_2)$.
%
%
\end{proof}

\

\subsection{Estimates for $\tLN$}
\begin{lemma}\label{est for tL_N}
Suppose the sequences $\eta_1,\eta_2 \in {\mathbb{W}}$ satisfy \eqref{1A}. Then:
\sline (i)
If $\eta_1 \in \Wqc$, then $\tLN \leq 1 + 2 \OT_N(\eta_1, \eta_2)$.
\sline (ii)
If $\eta_1 \in \Wc$, then $\tLN \leq 1 + 2 S_N(\eta_1, \eta_2)$.
\end{lemma}

\begin{proof}[Sketch of a proof]
Repeat the argument from the preceding lemma with $z = P_B(x)$.
The term $II$ is estimated exactly as above, while for the term $I$ we proceed as follows
\Bea \label{Eq-5-6}
I = \|(I-P_{B\cup \Gamma})(x)\| &=& \|x - P_B(x) - P_{\Gamma \setminus B}(x-P_B(x))\| \nonumber \\
&\leq& \|x - P_B(x)\| + \|P_{\Gamma \setminus B}(x-P_B(x))\|  \nonumber \\
&\leq & (1+K_N)\| x - P_B(x)\|\,.
\Eea
Now use \eqref{Eq-5-6} in place of \eqref{Eq-5-4} to obtain
\[
\tLN\leq 1+ 2\, S_N(\eta_1,\eta_2),
\]
or a similar estimate with $\OT_N(\eta_1,\eta_2)$ if we assume $\eta_1\in\Wqc$.
\end{proof}

\subsection{Estimates for $\tLN^*$ and $\LN^*$}
These can now be obtained applying the previous estimates to the system $\{\be^*_n,\be_n\}$, after interchanging the roles of $\eta_1$ and $\eta_2$ (and using the property $\eta_2\in\Wqc$ or $\eta_2\in\Wc$, respectively).

This completes the proof of all the asserted inequalities in Theorem \ref{Th3}, namely  \eqref{Th1a} and \eqref{Th1b}.
The optimality of the inequalities is illustrated with an example in $\S\ref{Example1}$ below.
\ProofEnd

\subsection{First corollaries}

%

\begin{corollary}\label{CorNuew}
If $c_1=\sup_n\|\be_n\|$ and $c_2=\sup_n\|\be^*_n\|_*$, then
\Be\OT_N(D,D^*) \,\leq \,\min\big\{ c_2\,D(N), c_1\,D^*(N)\big\}\,\leq\,c_1c_2\,N.\label{DDsN}\Ee
\end{corollary}

\begin{proof}
Using $D(j) \leq c_1 j$ (and $D^*\in\SW$), we deduce
\[
\Ts T_N(D,D^*)\leq c_1 \sum_{j=1}^N \Delta D^*(j) = c_1 D^*(N).\]
Changing the roles of $D$  and $D^*$ the result follows easily.
\end{proof}

\BR
Inserting \eqref{DDsN} in Theorem \ref{Th3} one recovers the classical bound $\LN\leq1+3c_1c_2N$; see e.g.  \cite[Theorem 1.8]{BBG2016}.
\ER

The next corollary could be applied quickly in some practical situations.

\begin{corollary} \label{Cor-5.2}
Let $\Been$ be a   complete 
s-biorthogonal system in $\SX$. Then
	
	\sline $i)$ $\;\max\lbrace \LN , \LN ^* \rbrace \lesssim \min\{D(N), D^{*}(N)\}\,.$
	
	\sline $ii)$ If $\;\min\{D(N), D^{*}(N)\} \lesssim K_N\;$, then $\;\LN \approx K_N \approx \min\{D(N), D^{*}(N)\}\,.$
	
	\sline $iii)$ If $\;d(N)\approx 1\;$, then $\;\LN\approx D(N)\, .$
\end{corollary}

\begin{proof}
	$i)$ follows from Theorem \ref{Th3} and the previous corollary.
	
	$ii)$ follows from (i) and the known lower bound $\LN\gtrsim K_N$; see e.g. \cite[Proposition 3.3]{GHO2013}.
	
	$iii)$ Finally, if $d(N) \approx 1$, then the superdemocracy parameter in \eqref{sdemo} takes the form
	$\mu_N=\sup_{n\leq N}D(n)/d(n)\approx D(N)$. So the result follows from (i)  and the known lower bound $\LN\gtrsim \mu_N$; see \cite[Proposition 1.1]{BBG2016}.
\end{proof}

\medskip

\section{Estimates for $D^*(N)$}\label{superdemocacydual}

In practice, Theorem \ref{Th3} needs good bounds of the upper democracy sequences
 $D(N)$ and $D^*(N)$, associated with $\cB$ and $\mathcal B^*$.
Sometimes the dual norm $\|\cdot\|_*$ is not explicit, or is hard to compute.
In this section we give bounds for $D^{*}(N)$ which only involve parameters of $\mathcal B$, namely the \emph{lower superdemocracy constants}, $d(N)$ or $\ld(N)$, defined in \eqref{dN},
and the  \emph{quasi-greedy constants} 
$$   
g_N = \sup_{n\leq N}\|G_n\|,\,\quad
g^c_N = \sup_{n\leq N} \|I-G_n\|,\mand \hg_N=\sup_{{0\leq k\leq n\leq N}}\big\|G_n-G_k\|.
$$  
Note that, by the triangle inequality, $\hg_N \leq 2 \min \{ g_n , g^c_N \}$.

\begin{proposition} \label{newlema}
	Let $\Been$ be a  complete
s-biorthogonal system in $\SX$. Then
	\Be \frac{N}{\ld(N)}\, \leq\, D^{*}(N) \,\leq \, \sum_{j=1}^N \frac{\hg_j}{d(j)}\,.\label{Dgd}\Ee
\end{proposition}

The proof is a slight generalization of known arguments from  \cite[Proposition 4.4]{DKKT} and \cite[Theorem 5]{Wo3} (see also \cite[Theorem 4]{Bed2008}).
We first recall another result from \cite{DKKT} (with the notation given in \cite[Lemma 2.3]{BBG2016}).

\begin{lemma} \label{L-7-1}
	Let $\{\be_n,\be_n^*\}_{n=1}^\infty$ be a  complete
		s-biorthogonal system in $\SX$. If $x\in \SX$, $ \Lambda \in \mathcal G(x,N)$ and $\e_n =\sign\,(\be_n^*(x))$, then
	$$
	\min_{n\in \Lambda} |\be_n^*(x)|\, \|\bone_{\bfe \Lambda}\| \leq \hg_N\, \|x\|\,.
	$$
\end{lemma}

\Proofof{Proposition \ref{newlema}}
	The left hand side of \eqref{Dgd} was shown in Lemma \ref{p0bis}.
	For the right inequality, we pick $|A|=N$ and $\bfe\in\cY$, and we shall estimate $\|\bone_{\bfe A}^*\|_* = \sup_{\|x\|=1} |\bone_{\bfe A}^*(x)|$. Take $x\in \SX$ with $\|x\|=1$, and let $\pi$ be a greedy ordering of $x$. Then
	\Beas
	    |\bone^*_{\bfe A}(x)| & = & \Big|\sum_{n\in A} \e_n \be_n^*(x)\Big| \leq  \sum_{n\in A} |\be_n^*(x)| \\
	   &\leq& \sum_{j=1}^N |\be_{\pi(j)}^*(x)| =  \sum_{j=1}^N |\be_{\pi(j)}^*(x)| \, \frac{\|\bone_{\bdt \Lambda_j}\|}{\|\bone_{\bdt \Lambda_j}\|}\,,
	\Eeas
	where $\Lambda_j \in \mathcal G(x, j)$ is a greedy set for $x$ of size $j$ and $\bdt = \{\sign\,(\be_n^*(x))\}$. By Lemma \ref{L-7-1} and $\|x\|=1$,
	$$
	|\bone^*_{\bfe A}(x)| \leq \, \sum_{j=1}^N  \, \frac{\hg_j}{\|\bone_{\bdt \Lambda_j}\|} \leq \, \sum_{j=1}^N \, \frac{\hg_j}{d(j)}\,.
	$$
Taking the sup over all $\|x\|=1$, $|A|=N$ and $\bfe\in\cY$ gives the desired result.
\ProofEnd

As special cases we obtain the following.


\begin{corollary} \label{Prop-7-2}
	Let $\Been$ be a  complete
		s-biorthogonal system in $\SX$.
	
\sline (i) If $\mathcal B$ is quasi-greedy then
${\ld}'(N)\, \leq\, D^{*}(N) \,\lesssim\, \widetilde{(d')}(N)\,.$
If additionally $d\in\Wqc$, then  $D^*(N)\lesssim\, d^\prime (N) \ln (N+1)$, and if $I_d<1$ then $\;D^*(N) \approx d^\prime(N) $.

\sline (ii)  If $\mathcal B$ is superdemocratic  then
$\;d'(N)\, \leq\, D^{*}(N) \,\lesssim\, (\sum_{j=1}^N\frac{g_j}{j})\, d^\prime(N)$. If additionally  $i_g>0$, then
$\;d'(N)\,\lesssim\,D^*(N) \lesssim \,g_N\,d^\prime(N) $.



\end{corollary}

\begin{proof} (i) is direct from \eqref{Dgd} and $\sup_N \hg_N<\infty$, and for the second part, from \eqref{tetalog} and Corollary \ref{cor2}. In (ii) one uses $d(N)\approx D(N)\in \Wqc$ in \eqref{Dgd}, together with Proposition \ref{p2}. 
\end{proof}

 We conclude with a new definition, which we find appropriate in this context.

 \begin{definition} We say that $\Been$ has the property $\bDs$ if $D^*(N)\approx d'(N)$.
 \end{definition}

We list various examples where this property holds (or fails).
\Benu
\item All bidemocratic bases (as in \eqref{bidem}) have the property $\bDs$.

\item All quasi-greedy bases with $d\in\Wqc$ and $I_d<1$ have the property $\bDs$, by Corollary \ref{Prop-7-2}.i.

\item Property $\bDs$ may fail when $I_d=1$, even for greedy bases.
 Indeed, the canonical system in the discrete Triebel-Lizorkin space $\mathfrak{f}_1^q$, $1\leq q\leq\infty$, is a greedy basis with $d(N)\approx D(N)\approx N$. However, using the duality between $\mathfrak{f}_1^q$ and  $bmo_{q'}$, one can show that $D^*(N)\approx [\ln (N+1)]^{1/q'}$.
\item The canonical basis in $\ell^p\oplus\ell^q$ has the property $\bDs$, for all $1\leq p,q\leq \infty$.
In fact, $d(N)\approx N^{\frac1p\wedge \frac1q}$, so $d'(N)\approx N^{\frac1{p'}\vee \frac1{q'}}\approx D^*(N)$.
\item The trigonometric system in $L^p(\ST)$ has the property $\bDs$ when $1<p\leq\infty$; see $\S\ref{Example3}$ below. However, this property fails for $p=1$, since $d'(N)\approx N/\ln (N+1)$, but $D^*(N)\approx N$.
\Eenu

\section{Corollaries in special cases}

 In this section we investigate the growth of $\OT_N (D,D^*)$ when $\cB$ is quasi-greedy, superdemocratic, or has property $\bDs$.
In all these cases we show that $\LN\lesssim \OT_N (D,D^*) \lesssim \LN\ln (N+1) $, so the loss in Theorem \ref{Th3} is at most logarithmic. 

\begin{lemma}\label{L_TDdg}
Let $\Been$ be a  complete
	s-biorthogonal system in $\SX$. Then
\Be \label{TNgmu}
  T_N  (D, D^*) \leq  \sum_{j=1}^N \frac{\hg_j\,\mu_j}{j}\,.
    \Ee
\end{lemma}
\begin{proof} By Proposition \ref{newlema}, $ D^*(N)\leq  \sum_{j=1}^N \frac{\hg_j}{d(j)}=:\eta(N)$. Using $D\in\Wqc$ and Lemma \ref{etaxi} it follows that
\[
T_N(D,D^*)\leq T_N(D,\eta)=   \sum_{j=1}^N \frac{D(j)}{j}\frac{\hg_j}{d(j)}\,
\leq  \sum_{j=1}^N \frac{\hg_j\,\mu_j}{j}\,.
\]
\end{proof}

\begin{corollary} \label{pro-7-4-new}
 Let $\Been$ be a  complete
	s-biorthogonal system in $\SX$. If $\mathcal B=\{\be_n\}_{n=1}^\infty$ is superdemocratic, then
\Be
\max\big\{K_N,\tLN,\LN,g^*_N,\mu^*_N,\tLN^*,\LN^*\big\} \,\lesssim T_N(D,D^*)\lesssim \,g_N\,\ln(N+1).\label{maxKLg}
\Ee
In particular,
\Be
 \LN  \lesssim \OT_N  (D,D^*) \lesssim \tLN  \ln(N+1)\,, \quad N=1,2,\dots
\label{LNlog}\Ee
Finally, if $i_g>0$ then, $\OT_N\approx\LN\approx \tLN\approx K_N\approx g_N$.
\end{corollary}

\begin{proof}
Call $C_s:=\sup_N\mu_N<\infty$. Then \eqref{TNgmu} gives
    \Be
  T_N  (D,D^*) \leq C_s\, \sum_{j=1}^N \frac{\hg_j}{j} \leq\,C_s\, \hg_N (1+\ln N).
    \label{TN_Cs}\Ee The results now follow easily from Theorem \ref{Th3}
and the known lower bounds $\LN\geq\tLN\geq g^c_N$ and $\tLN^*\gtrsim\max\{g^*_N,\mu^*_N\}$; see e.g. \cite[Prop 1.1]{BBG2016}.
Apply Proposition \ref{p2} to \eqref{TN_Cs} in order to handle the case of $i_g > 0$.
\end{proof}

\BR\label{Rsdem}
From \eqref{maxKLg} we see that, for superdemocratic bases,
\Be
K_N\,\lesssim\, g_N\,\ln(N+1),
\label{kgN}\Ee
that is, $K_N/g_N$ cannot grow arbitrarily. This was known for quasi-greedy bases
\cite[Lemma 8.2]{DKK}, but seems to be unnoticed for general superdemocratic bases.
\ER

\BR Remark 4.6 in \cite{DKKT} provides an example of a superdemocratic basis with $i_g>0$,
which is neither quasi-greedy nor bidemocratic.
Our result implies the asymptotically optimal bound $\LN \approx \tLN\approx K_N\approx g_N$.
\ER

 \begin{corollary}\label{pro-new1}
  Let $\Been$ be a  complete
 	s-biorthogonal system in $\SX$.	
 Assume that either $\cB=\{\be_n\}_{n=1}^\infty $ is quasi-greedy, or $\Been$ has the property $\bDs$. Then 
\Be
\max\big\{K_N,\tLN,\LN,g^*_N,\mu^*_N,\tLN^*,\LN^*\big\} \,\lesssim T_N(D,D^*)\lesssim\,\mu_N\,\ln(N+1).\label{maxKLmu} \Ee
In particular, \eqref{LNlog} holds, and moreover,
 \Be
\mu_N\,\lesssim\,\tLN\,\leq	\,\LN\,  \lesssim \mu_N\ln(N+1)\,, \quad N=1,2,\dots
\label{LLmuN}\Ee
Finally,  if $i_\mu>0$, then $\LN\approx \tLN\approx \mu_N$.
\end{corollary}
\begin{proof}
 (i) If  $C_q=\sup_{j\geq1}\hg_j<\infty$, then \eqref{TNgmu} gives
   \Be
    T_N  (D,D^*) \leq \,C_q\, \sum_{j=1}^N \frac{\mu_j}{j} \leq\,C_q\, \mu_N \,(1+\ln N)\,.
    \label{Tmulog}\Ee
The assertions now follow from Theorem \ref{Th3} and the lower bounds in \cite[Prop 1.1]{BBG2016}.
		
		\sline (ii) Assuming property $\bDs$, and using that $D^*\in \Wqc$, one has
		\Be D(j)\Delta D^*(j)\le D(j)\frac{ D^*(j)}{j}\approx
D(j)/d(j)\le \mu_j, \quad j\in \mathbb N .\label{DDsmu}\Ee
Thus, also in this case we deduce $T_N  (D,D^*)\lesssim\sum_{j=1}^N\mu_j/j\leq \mu_N(1+\ln N)$.
\end{proof}

\

As a consequence we obtain a criterion for
$\OT_N  (D,D^*) \lesssim \ln(N+1)$, which includes in particular all greedy bases.

\begin{corollary} \label{pro-7-5-new}
 Let $\Been$ be a complete
		s-biorthogonal system in $\SX$. If $\mathcal B=\{\be_n\}_{n=1}^\infty$ is almost greedy, or
	 $\Been$ is bidemocratic, then
	\Be
\max\big\{K_N,\tLN,\LN,g^*_N,\mu^*_N,\tLN^*,\LN^*\big\} \lesssim \OT_N  (D, D^*) \lesssim \ln(N+1)\,.
\label{LLs}\Ee
\end{corollary}

\begin{proof}
This follows from \eqref{maxKLmu}, using $\mu_N\approx 1$.
\end{proof}

\

We pose two questions.

\bline\textbf{Question 1:} Characterize the systems $\{\be_n,\be^*_n\}$ for which $\OT_N(D,D^*) \lesssim \LN  \ln(N+1)$.

\bline\textbf{Question 2:} Characterize the systems for which $ \max\{\LN ,\LN^*\} \lesssim \ln(N+1).$

\

Concerning Question 1, all the examples we have tested seem to satisfy this property.
Concerning Question 2, $\OT_N (D,D^*) \lesssim\ln (N+1)$ gives a sufficient condition, but we do not know whether it is necessary.

\section{Examples} \label{Examples}

In this section we give explicit examples which  illustrate the essential sharpness of our previous results.

%

\subsection{Example 1: The difference basis in $\ell^1$} \label{Example1}
Let $\{\be_n\}_{n=1}^\infty$ denote the canonical basis in $\ell^1(\SN)$, and consider the system
\Be \label{Eq-8-1}
\bx_1 = \be_1\,, \quad \bx_n = \be_n - \be_{n-1}\,, \ n=2,3, \dots
\Ee
This is a monotone basis in $\SX=\ell^1$, sometimes called the \textbf{difference basis}.
Observe that for finitely supported real scalars $\{b_n\}_{n=1}^\infty$ one has
\Be \label{Eq-8-2}
\Big\| \sum_{n=1}^\infty b_n \bx_n\Big\| = \sum_{n=1}^\infty |b_n - b_{n+1}|\,.
\Ee
In particular, $\|\bx_1\|=1$ and $\|\bx_n\|=2$ if $n\geq 2$. The dual system consists of the $\ell^\infty$-vectors $\bx_n^*=\sum_{m=n}^\infty\be_m^*$,
so for  $\{c_n\}\in c_{00}$ it holds that
\Be \label{Eq-8-6}
\Big\| \sum_{n=1}^\infty c_n \bx_n^*\Big\|_* =
\sup_{n\geq 1} \Big|\sum_{j=1}^n c_j\Big|\,.
\Ee
The system $\{\bx_n^*\}_{n=1}^\infty$ is called the \textbf{summing basis}; see e.g. \cite[p.20]{LZ}.

\begin{lemma} \label{L8-1}
	For $\{\bx_n,\bx^*_n\}_{n=1}^\infty$ as above and  $\,N=1,2,3, \dots$, we have
	\sline (i) $d(N) =1 \quad \mbox{and} \quad D(N) =2N\,$
	\sline (ii) $d^*(N)=1\quad$ and $\quad D^*(N)=N$.
\end{lemma}

\begin{proof}
	For $A \subset \mathbb N, |A|=N$ and $\bfe\in \cY = \{\pm 1\}$, if follows from (\ref{Eq-8-2}) that
	\Be
	1 \leq \|\bone_{\bfe A}\| = \Big\|\sum_{n\in A} \e_n \bx_n \Big\| \leq 2N\,.
	\label{1Adiff}\Ee
	Using again \eqref{Eq-8-2}, it is easily seen that the right equality in \eqref{1Adiff} is attained by testing with $\sum_{j=1}^N \bx_{2j}$, while
	the left equality is attained with $\sum_{j=1}^N \bx_{j}$. This shows the statements in (i). The statements in (ii) about the summing bases are similar (and can also be found in \cite[Example 5.1]{BBG2016}).
\end{proof}

\begin{proposition} \label{Pro-8-3}
The system $\{\bx_n,\bx_n^*\}_{n=1}^\infty$ satisfies $S_N(D,D^*)=\OT_N(D,D^*)=2N$. Moreover,
 $$K_N =K_N^*=2N\,, \quad \tLN  =\tLN ^*= 1+4N\,, \quad \mbox{and} \quad \LN  = \LN ^*=1+6N\,.$$
In particular, equalities are attained everywhere in Theorem \ref{Th3}.
\end{proposition}

\begin{proof}
From Lemma \ref{L8-1} we have
	$$
	T_N(D,D^*)=\sum_{j=1}^N\frac{D(j)}{j}\Dt D^*(j)=2N=T_N(D^*,D)=S_N(D,D^*)\,,
	$$
establishing the first assertion. Theorem \ref{Th3} then implies\[
K^*_N\leq K_N\leq 2N, \quad \tLN,\tLN^* \leq 1+4N,\mand \LN ,\LN^* \leq 1+6N.\]
The equalities for $K_N^*$, $\tLN^*$ and $\LN^*$ were shown in \cite[Proposition 5.1]{BBG2016}.
We show here that equalities are attained also for $\tLN$ and $\LN$.
First consider
	$$
	x = \sum_{j=1}^{2N+1} \bx_j + \sum_{j=2N+1}^{3N} \bx_{2j}\,.
	$$
	Then, $
	\widetilde \sigma_N(x) \leq \Big\| \sum_{j=1}^{2N+1} \bx_j\Big\| = 1$.
	However, choosing $G_Nx=\sum_{j=1}^{N} \bx_{2j}$ we have
	\Beas
	\|x-G_Nx\| &=& \Big\|\sum_{j=1}^{N+1} \bx_{2j-1} + \sum_{j=2N+1}^{3N} \bx_{2j} \Big\|
=  4N+1\,.
	\Eeas
	Therefore,
	$	\tLN  \geq {\|x-G_Nx\|}/{ \widetilde \sigma_N(x)} \geq4N+1\,$.	Finally, consider
	$$
	x = \bx_1 + \sum_{j=1}^{N} \bx_{4j-2} + \sum_{j=1}^{N} \bx_{4j-1} - \sum_{j=1}^{N} \bx_{4j} + \sum_{j=1}^{N} \bx_{4j+1}\,.
	$$
	Taking $ G_Nx=\sum_{j=1}^{N} \bx_{4j-2}$ we obtain $\|x-G_Nx\| = 1+6N.$ On the other hand, choosing $y=2\,\sum_{j=1}^{N} \bx_{4j} \in \Sigma_N$, we have \[\sigma_N(x)\,\leq\,\|x+y\|=\big\| \sum_{j=1}^{4N+1} \bx_{j}\big\|\,=\,1.\]
Thus, $\LN \geq {\|x-G_Nx\|}/{ \sigma_N(x)} \geq 1+6N$.
\end{proof}

\medskip

\subsection{Example 2: The Lindenstrauss basis and its dual} \label{Example2}

Let $\{\be_n\}_{n=1}^\infty$ denote the canonical basis in $\ell^1(\SN)$, and consider the vectors
\Be \label{Eq-8-7}
\bx_n = \be_n - \frac12\,\be_{2n+1} - \frac12\,\be_{2n+2}\,, \qquad n=1,2,3, \dots
\Ee
The system $\mathcal L = \{\bx_n\}_{n=1}^\infty$ was introduced by J. Lindenstrauss in \cite{L}.
It is a basic sequence of $\ell^1$, hence a basis of a subspace $\mathbb D=\Cspan{\mathcal L}$ in $\ell^1$.
To describe the dual system we consider the following vectors in $c_0$:
\Be
\by_n:=\sum_{j=0}^n 2^{-j}\be_{\ga_j(n)} ,\quad n=1,2,3,\ldots
\label{byn}\Ee
where $\ga_0(n)=n$ 
and $\ga_{j+1}(n)=\lfloor \frac{\ga_j(n)-1}2\rfloor$, $j \geq 0$
(with the convention
$\be_\ga=\bz$ if $\ga\leq0$). It is shown in \cite[Example 2]{HR1970} that  $\mathcal{Y}=\{\by_n\}_{n=1}^\infty$ is a Schauder basis in $c_0$ with dual vectors $\by^*_n=\bx_n$. In particular, there exists some $c>0$ such that
\[
c\,\|y\|_{c_0} \leq \sup_{{x\in\SD}\atop{\|x\|_{\ell^1}=1}} |\langle x,y\rangle| =\|y\|_{\SD^*}\,\leq \,\|y\|_{c_0},\quad y\in c_0\,;
\]see e.g. \cite[Exercise 6.12]{Fabian}. So we can identify $\widehat{\SD}$ and $c_0$ with equivalent norms.
We summarize a few other properties of the biorthogonal pair $\{\cL,\mathcal{Y}\}$.
\begin{itemize}
	\item
	$\mathcal L$ is conditional in $\mathbb D$, and $\mathbb D$ has no unconditional basis; \cite[p. 454-457]{Sin}.
	\item
	$\mathcal L$ is a quasi-greedy basis in $\mathbb D$, with $\sup_{N\geq1}\|G_N\| \leq 3$; see \cite{DM}.
	\item
	$\mathcal Y$  is not quasi-greedy in $c_0$; see \cite{DM}.
	\item  $K_N(\cL, \SD) \approx \ln(N+1)\,, N= 1, 2, 3, \dots$; see\footnote{This is shown in \cite{GHO2013} for the system $\{\be_n - (\be_{2n} +\be_{2n+1})/2\}_{n=1}^\infty$, but the same arguments, with obvious modifications, work for the basis in (\ref{Eq-8-7}).} \cite[$\S6$]{GHO2013}.
\end{itemize}

\begin{theorem}\label{ThLind}
 For the Lindenstrauss basis $\mathcal L$ in $\SD$ we have $\OT_N(D,D^*)\approx\ln(N+1)$. Moreover,
\Be
\tLN\approx 1,\mand \LN\approx \LN^*\approx\tLN^*\approx K_N\approx g^*_N\approx\mu^*_N\approx \ln(N+1).
\label{Lindbounds}\Ee
\end{theorem}

\BR The results for the system $\mathcal Y$ seem to be new. In fact, in this example, Theorem \ref{Th3} 
performs better than Theorems 1.2 and 1.3 from \cite{BBG2016}, which would only yield the non-optimal bound
$\LN(\mathcal{Y},c_0)\lesssim [\ln (N+1)]^2$.
\ER

We only need upper  estimates for $D$ and $D^*$, but we shall actually prove more.

\begin{lemma} \label{Lemma8-4} For the Lindenstrauss basis $\mathcal L$ in $\SD$ we have the following

\sline (i)    $  d(N)\approx N\;$ and $\;D(N) =2N.$

\sline (ii)  $d^*(N)\approx 1\;$ and $\;D^{*}(N) \approx \ln (N+1).$
\end{lemma}

\begin{proof}
	$i)$
	Let  $\bone_{\bfe A} = \sum_{n\in A} \e_n \bx_n\,,$ with $|A|=N$, $\bfe\in \cY$. Since $\|\bx_n\|=2,$ one always has $\|\bone_{\bfe A}\| \leq 2N$. To see that this bound is attained consider $$ \bx= \sum_{n=1}^N \bx_{3^n} = \sum_{n=1}^N (\be_{3^n} - \frac12\,\be_{2\cdot 3^n + 1} - \frac12\,\be_{2\cdot 3^n + 2}\,).$$
Since $2\cdot 3^n + 2 < 3^{n+1}$, one deduces that $\|\bx\| = 2N$. Hence $D(N)=2N.$

We now give a lower estimate for $d(N)$. Observe that
\[
\Big\|\sum_{n=1}^M b_n\bx_n\Big\|_{\ell^1}\,=\,|b_1|+|b_2|+
\sum_{n=3}^M\big|b_n-\tfrac12b_{\lfloor\frac{n-1}2\rfloor}\big|\,+\tfrac12\sum_{n=M+1}^{2M+2}\big|b_{\lfloor\frac{n-1}2\rfloor}\big|.
\]
From here it easily follows that $\|\bone_{\bfe A}\|_{\ell^1}\geq |A|/2$, since for $n\in A$ we have $|b_n-\tfrac12b_{\lfloor\frac{n-1}2\rfloor}|\geq 1/2$. Thus\footnote{Slightly more elaborate computations actually lead to $d(N)=N+1$.} \Be
N/2\leq d(N) \leq 2N.\label{dN2}\Ee

	$ii)$ Using \eqref{dN2} and $g_N\leq3$ in Proposition \ref{newlema} yields
	\Be \label{Eq-8-7-1}
	 D^{*} (N)  \leq \,C\ln (N+1).
	\Ee
	The reverse inequality, $D^*(N)\gtrsim\ln(N+1)$ follows from
	\Be \label{Eq-8-8}
	\Big\|\sum_{i=1}^{2^{N+1}-2} \by_i \Big\|_{*} \geq \frac{N}{2}\,;
	\Ee
see (10) in \cite{DM}. To estimate $d^*$ we quote the equality (9) in \cite{DM},
	\Be \label{Eq-8-9}
	\Big\|\sum_{i=1}^{2^{N+1}-2} (-1)^i \,\by_i \Big\|_{c_0}= 1\, .
	\Ee
Since $\mathcal Y$ is a Schauder basis,
this actually implies that $d^{*} (N)\lesssim1$. On the other hand, given any $A\subset\SN$, if we set $n_0=\min A$, then  \[
\|\sum_{n\in A}\e_n\by_n\|_{c_0} \gtrsim \|\by_{n_0}\|_{c_0}=1,
\]
which implies\footnote{Slightly more elaborate computations, using the definition of $\by_n$ in \eqref{byn}, actually give $d^*_{c_0}(N)=1$, and 
also $D^*_{c_0}(N)=\log_2(N+1)$ if $N+1=2^n$.}  $d^{*} (N)\gtrsim 1$.
\end{proof}

\Proofof{Theorem \ref{ThLind}}
\hskip-.5cm By Lemma \ref{Lemma8-4} we have   $D(j) = 2j$, and therefore,
\Be
S_N(D, D^*)=T_N(D, D^*) = 2D^*(N) \approx \ln(N+1).
\label{TN_Lind}\Ee
Thus, Theorem \ref{Th3} gives a logarithmic upper bound for all the quantities in \eqref{Lindbounds}. Also, $\tLN\approx 1$ is known from \cite{DM} (since $\cL$ is quasi-greedy and democratic).

For the lower bounds, first note that $\LN\gtrsim K_N\gtrsim\ln (N+1)$ was shown in \cite[$\S6.1$]{GHO2013}.
Lemma \ref{Lemma8-4} also gives $\mu^*_N\approx \ln (N+1)$. Finally,
$\LN^*\geq\tLN^*\gtrsim g^*_N$, and the estimate $g^*_N\gtrsim \ln (N+1)$ can easily be obtained from \eqref{Eq-8-8} and \eqref{Eq-8-9}.
\ProofEnd

\medskip
	
\subsection{Example 3: The trigonometric system in $L^p(\mathbb T^d)$} \label{Example3}

Consider the system $\mathcal T^d = \{e^{2\pi i k\cdot x}\}_{k\in \mathbb Z^d}$ in the Lebesgue space $L^p(\mathbb T^d), 1 \leq p < \infty$, or in  $C(\mathbb T^d)$ when $p=\infty$. Temlyakov proved in \cite{Tem98trig} that $\LN\approx N^{|\frac1p-\frac12|}$. Here we recover this result as an application of Theorem \ref{Th3} (at least if $p\not=2$).

\begin{proposition} \label{Pro-8-8} For the system $\mathcal T^d$ in $L^p(\mathbb T^d)$ with $1\leq p\leq \infty$,  $p\not=2$, we have
\Be
\OT_N(D,D^*)\approx \LN \approx \tLN\approx K_N \approx\LN^*\approx\tLN^*\approx  N^{|\frac1p-\frac12|}.
\label{LN_Lp}\Ee
\end{proposition}

\begin{proof}
From the Hausdorff-Young inequality and elementary inclusions, it is straightforward to prove that
	\Be
N^{\frac 12\wedge\frac1{p'}}\,\leq	    \| \bone_{\bfe A}\|_p \leq  N^{\frac12\vee \frac1{p'}},
	\Ee
for all $|A|=N$  and $\bfe\in\cY$.
Thus,  \[
D(N)\leq N^{\frac12\vee \frac1{p'}}\mand D^*(N)\leq N^{\frac12\vee \frac1{p}},
\]
and therefore
$$
	\OT_N(D,D^*) \leq U_N (D,D^*) = \sum_{j=1}^N \frac{j^{|\frac{1}{p} -\frac{1}{2}|}}j \leq c_p
	N^{|\frac{1}{p} -\frac{1}{2}|}\,,
	$$
with $c_p=1/|\frac1p-\frac12|$. This and Theorem \ref{Th3} provide upper bounds for the constants in \eqref{LN_Lp}. The lower bounds follow from $g_N\gtrsim N^{|\frac{1}{p} -\frac{1}{2}|}$; see \cite[Remark 2]{Tem98trig}.
\end{proof}
\BR
When $\SX=L^2$ one of course has $K_N=\tLN=\LN=1$. 
Observe, however, that $D(j)=D^*(j)=\sqrt j$ only gives $\OT_N\approx \ln(N+1)$. 
This loss is due to the fact that, in Theorem  \ref{Th3}, we only make use of the weak assumptions $\ell^{2,1}\hookrightarrow\SX\hookrightarrow\ell^{2,\infty}$, rather than the full force of $\SX=\ell^2$.
\ER

\subsection{Example 4: A summing basis by blocks.}\label{Example4} This is a slight modification of an example exhibited in \cite[Proposition 7.1]{GHO2013}. It again  illustrates that Theorem \ref{Th3} produces asymptotically optimal bounds, which cannot be obtained with the results in \cite{BBG2016}. Take any $\{\omega_j\}_{j=1}^\infty\in \SW_{\rm qc}$, say with $\om_1=1$. Define a space $\SX$ consisting of (real) sequences $x=(x_n)_{n=1}^\infty\in c_0$ such that
$$\Vert x\Vert = \max\Bigg\lbrace \Vert x\Vert_\infty, \ \underset{j\geq 1}{\sup}\,\underset{N\geq 1}{\sup}\, \frac{\omega_j}{j}\Big|  \sum_{{n\in\Delta_j}\atop{n\leq N}}x_n\Big|\Bigg\rbrace\,<\,\infty,$$
where $\Delta_j= \lbrace 2^j,...,2^j + 2j-1\rbrace$, $j=1,2,...$ By definition of the norm, the canonical system $\mathcal{B} = \lbrace \textbf{e}_n\rbrace_{n=1}^\infty$  is a monotone basis in $\SX$, with $\|\be_n\|=\|\be^*_n\|_*=1$ for all $n$.

\begin{proposition} In this example we have $\OT_N(D,D^*)\leq 2\om_N$, and therefore
\Be
K_N\leq 2\om_N,\quad  \tLN\leq1+4\om_N,\mand \LN \leq1+6\om_N,\qquad N=1,2,\ldots
\label{gNomN}\Ee
Moreover, all these quantities are bounded below by $\min\{g_N,g^c_N\}\geq\om_N$.
\end{proposition}
\Proof
For any $|A|=N$ and $\bfe\in\cY$  we claim that
\Be 1\leq \Vert \bone_{\bfe A}\Vert\leq \|\bone_A\|= \max \left\lbrace 1,
\sup_j\,\frac{\om_j}j\,|\Delta_j \cap A|\right\rbrace \leq 2\omega_{N}.\label{D_omj}\Ee
Indeed, the last inequality is justified using the quasi-concavity of $\om$ as follows:
\begin{itemize}
	\item if $j\geq N$, then $\frac{\om_j}j\,|\Delta_j \cap A| \leq \frac{\om_j}j\,|A|=\frac{\om_j}{j}\,N\leq\om_N$
	\item if $j\leq N$, then $\frac{\om_j}j\,|\Delta_j \cap A|\leq \frac{\om_j}j\,|\Dt_j|=2\om_j\,\leq\,2\om_N$.
\end{itemize}
On the other hand, we have the trivial estimate $\Vert \bone_{\bfe A}^*\Vert_* \leq \vert A\vert$.
Therefore,  arguing as in Corollary \ref{CorNuew} we obtain $\OT_N(D,D^*)\leq 2\om_N$, and therefore \eqref{gNomN}.
 We now show the lower bound.
Let $x=\sum_{j=0}^{2N-1}(-1)^j\be_{2^N+j}$, which has support in $\Dt_N$ and $\|x\|=1$.
Choosing $G_Nx=\sum_{\ell=0}^{N-1}\be_{2^N+2\ell}$, we see that
\[
g_N\geq \|G_Nx\|=\om_N\mand g_n^c\geq \|(I-G_N)x\|=\om_N.
\]
\ProofEnd

\subsection{Example 5: An example of Konyagin and Temlyakov} \label{Example5}

We slightly generalize a construction in \cite{KT} of a quasi-greedy superdemocratic basis which is not unconditional. For $1\leq p<\infty$ and $1\leq r\leq \infty$, let $KT(p,r)$ be the set of all sequences $\bx=\{x_n\}_{n=1}^\infty\in c_0$ with norm
$$\vertiii{\bx}=\max\big\lbrace \|\bx\|_{\lpr},\;\|\bx\|_{\bp}\big\rbrace\,<\,\infty$$
where
\[
\|\bx\|_{\lpr}=\left(\sum_{j=1}^\infty (j^{1/p}x_j^*)^r \frac{1}{j}\right)^{1/r},\mand \|\bx\|_{\bp}=\sup_{N\geq1}\Big|\sum_{n=1}^N\frac{x_n}{n^{1/p'}}\Big|.
\]
The example in \cite[$\S3.3$]{KT} is the case $KT(2,2)$, while  $KT(p,p), 1<p<\infty, $ was later considered in \cite{GHO2013}.
A trivial case corresponds to $r=1$, for which $K(p,1)=\ell^{p,1}$.

We summarize the main results in the next theorem, where we write $\cB=\{\be_n\}_{n=1}^\infty$ for the standard canonical basis. 

\begin{theorem}\label{th_KT}
Let $1\leq r\leq \infty$. 
\sline (i) If $1<p<\infty$ then $\big(KT(p,r),\cB\big)$ is quasi-greedy, bidemocratic and
\Be
\LN\approx\LN^*\approx K_N\approx [\ln (N+1)]^{1/r'}\mand \tLN\approx\tLN^*\approx 1.
\label{LNKTp}\Ee
\noindent (ii) If $p=1$ then $\big(KT(1,r),\cB\big)$ is superdemocratic and
\Be
\OT_N(D,D^*)\approx \LN\approx \tLN\approx\LN^*\approx \tLN^*\approx K_N\approx g_N\approx\mu_N^*\approx [\ln (N+1)]^{1/r'}.
\label{LNKTpii}\Ee
\end{theorem}

We split the proof in various lemmas, starting with the computation of $D$ and $D^*$.

\begin{lemma}\label{DN_KTpr}
If $1\leq r\leq \infty$, the following holds for the space $KT(p,r)$:
\sline (i) If $1<p<\infty$, then $\; d(N)\,\approx\,D(N)\approx N^{1/p}$, and $\;d^*(N)\approx D^*(N) \approx N^{1/p'}$.
\sline (ii) If $p=1$, then $\;d(N)\,\approx\,D(N)\approx N$, $\; d^*(N)=1$ and $D^*(N)\approx [\ln (N+1)]^{1/r'}$.

\sline In particular, $\big(KT(p,r),\cB\big)$ is always superdemocratic, and is bidemocratic if $p>1$.
\end{lemma}
\Proof
If $|A|=N$ and $\bfe\in\cY$, then   \Be
\tri{\bone_{\bfe A}}\leq \tri{\bone_A}\leq
\max\Big\{[\sum_{j=1}^{N}(j^\frac1p)^r\tfrac1j]^\frac1r, \sum_{j=1}^{N}\tfrac1{j^{1/p'}}\Big\}=\sum_{j=1}^{N}\tfrac1{j^{1/p'}}\leq pN^{1/p},
\label{KTD_aux1}\Ee
and \[
\tri{\bone_{\bfe A}}\geq \|\bone_{\bfe A}\|_\lpr=[\sum_{j=1}^{N}(j^\frac1p)^r\tfrac1j]^\frac1r\geq\,c_{p,r}\, {N}^{1/p},
\]
for some 
 $c_{p,r}>0$. This shows that $d(N)\approx D(N)\approx N^{1/p}$ for all $1\leq p<\infty$.
 For the assertion about the dual system, observe that if $\tri{\bx}=1$, then
\Bea
|\bone^*_{\bfe A}(\bx)| & \leq &  \sum_{n\in A}|x_n|\;\leq\; \sum_{j=1}^N x^*_j\nonumber\\
& \leq & \|\bx\|_\lpr\,\Big[\sum_{j=1}^Nj^\frac{r'}{p'}\tfrac1j\Big]^\frac1{r'}\leq \left\{\Ba{ll} N^{1/p'} & \mbox{{\small if $1<p<\infty$}} \cr
[\ln(N+1)]^{\frac1{r'}} & \mbox{{\small if $p=1$}} \Ea\right.
\label{Ds_KT1r}\Eea
So taking sup over $\tri{\bx}=1$ we obtain the asserted upper bounds for $D^*(N)$.
For the lower bound, using \eqref{KTD_aux1}, \Be
\tri{\bone^*_{\bfe A}}_*\geq \bone^*_{\bfe A}(\bone_{\bar\bfe A})/\tri{\bone_{\bar{\bfe} A}}\geq N/(pN^{\frac1{p}})= N^{\frac1{p'}}/p.
\label{KTD_aux2}\Ee
So, when $1<p<\infty$ we have already proved $d^*(N)\approx D^*(N)\approx N^{1/p'}$. When $p=1$, one can obtain $d^*(N)=1$ from
\eqref{KTD_aux2} and
\[
\tri{\bone^*_{\{1,\ldots,N\}}}_*=\sup_{\tri{\bx}=1}\big|\sum_{n=1}^Nx_n\big|\leq1.
\]
Finally, setting $\e_n=(-1)^n$ and $\bx=\sum_{n=1}^N\frac{(-1)^n}n\be_n$, we have
$\tri{\bx}\approx [\ln (N+1)]^{1/r}$ and therefore
\[
\tri{\bone^*_{\bfe \{1,\ldots,N\}}}_*\geq \Ts \big|\sum_{n=1}^N\frac1n\big|/\tri{\bx} \,\approx\,[\ln (N+1)]^{1/r'}.
\]
This and \eqref{Ds_KT1r} show that $D^*(N)\approx [\ln (N+1)]^{1/r'}$, and establish the lemma.
\ProofEnd

The following proof is a variation of \cite[$\S3.4$]{KT}.

\begin{lemma}\label{qg1}
Let $1<p<\infty$ and $1\leq r\leq \infty$. Then $\cB$ is quasi greedy in $KT(p,r)$.
\end{lemma}

\begin{proof}
Since the canonical basis is unconditional in $\ell^{p,r}$ and $KT(p,1)=\ell^{p,1}$ we may assume that $r>1$.
Also, it suffices to show that $\|G_N\bx\|_{\bp}\leq C\,\tri{\bx}$, for all $G_N\in\mathcal{G}_N$ and all $N$.
Let $\bx\in {KT(p,r)}$, $\Lambda\in\mathcal{G}(\bx,N)$, $\alpha = \min_{j\in\Lambda}x_j^*$ and $M_\alpha = \left(\frac{\tri{\bx}}{\alpha}\right)^p\ge 1$.

 Then, for $M\leq M_\al$,  using that $\vert x_j\vert\leq \alpha$ if $j\in\Lambda^c$, we obtain
	\begin{eqnarray}\label{c1}
	\Bigg| \sum_{\underset{j\in\Lambda}{j=1}}^M \frac{x_j}{j^{1/{p'}}}\Bigg| &\leq& \Bigg\vert \sum_{j=1}^{M} \frac{x_j}{j^{1/{p'}}}\Bigg| + \Bigg\vert \sum_{\underset{j\in\Lambda^c}{j=1}}^{M} \frac{x_j}{j^{1/{p'}}}\Bigg| \leq \Vert \bx \Vert_\bp + \alpha\sum_{j=1}^{M_\alpha} \frac{1}{j^{1/{p'}}}\nonumber \\
	&\lesssim& \vertiii{\bx} + \alpha M_\alpha^{1/p} \lesssim \vertiii{\bx}.
	\end{eqnarray}	
For $M > M_\al$, we use \eqref{c1} to obtain
\Be\label{c2}
	\Bigg| \sum_{\underset{j\in\Lambda}{j=1}}^M \frac{x_j}{j^{1/{p'}}}\Bigg| \leq \Bigg| \sum_{\underset{j\in\Lambda}{j=1}}^{M_\alpha} \frac{x_j}{j^{1/{p'}}}\Bigg| + \Bigg| \sum_{\underset{j\in\Lambda}{M_\alpha < j\leq M}} \frac{x_j}{j^{1/{p'}}}\Bigg| \lesssim \vertiii{\bx} + \underbrace{\Bigg|\sum_{\underset{j\in\Lambda}{M_\alpha < j\leq M}} \frac{x_j}{j^{1/{p'}}}\Bigg|}_{(I)}.
	\Ee
To estimate $(I)$, take a number $q$ such that $\max\lbrace 1, p/r\rbrace<q<p$.
Set $s=rq/p>1$ (if $r = \infty$, then $s = \infty$ as well). By the Hardy-Littlewood rearragement inequality,
	\begin{eqnarray*}
		(I)&\leq& \sum_{j=1}^N \frac{x_j^*}{(j+M_\alpha)^{1/{p'}}}\leq  \alpha^{1-p/q} \sum_{j=1}^N\frac{(x_j^*)^{p/q}j^{1/q}j^{1/q'}}{(j+M_\alpha)^{1/{p'}}}\frac1j \\
		&\leq& \alpha^{1-p/q}\left(\sum_{j=1}^\infty (j^{1/p}x_j^*)^{sp/q}\frac{1}{j} \right)^{1/s}\left(\sum_{j=1}^\infty \left(\frac{j^{1/q'}}{(j+M_\alpha)^{1/p'}}\right)^{s'}\frac1j\right)^{1/s'}\\
		&\leq& \alpha^{1-p/q} \vertiii{\bx}^{p/q}\underbrace{\left(\sum_{j=1}^{\infty} \frac{j^{s'/{q'}}}{(j+M_\alpha)^{s'/p'}}\frac1j\right)^{1/s'}}_{(II)}.
	\end{eqnarray*}	
Finally, we estimate $(II)$ as follows:
\begin{eqnarray} \label{c3}
    (II) &\leq & M_\alpha^{-1/p'}\left( \sum_{j\leq M_\alpha} \frac{j^{s'/q'}}{j}\right)^{1/s'} + \left( \sum_{j>M_\alpha} \frac{1}{j^{(\frac1{p'}-\frac1{q'})s'}}\frac{1}{j}\right)^{1/s'} \nonumber\\
    &\lesssim & M_\alpha^{1/q' - 1/p'} \leq \big({\vertiii{\bx}}/{\alpha}\big)^{p(1/p - 1/q)}.
\end{eqnarray}
	Hence, using \eqref{c3} in the estimate of $(I)$,
	\begin{eqnarray}\label{I}
	(I)\lesssim \alpha^{1-p/q}\vertiii{\bx}^{p/q}\vertiii{\bx}^{1-p/q}/\alpha^{1-p/q} = \vertiii{\bx}.
	\end{eqnarray}
Thus \eqref{I}, \eqref{c2}, and \eqref{c1} show that $\|G_N\bx\|_{\bp}\lesssim\,\tri{\bx}$, establishing the result.
\end{proof}

\begin{lemma}\label{qg2}
	For $1\leq p<\infty$ and $1\leq r\leq \infty$, we have $K_N\gtrsim (\ln (N+1))^{1/r'}$.
	In particular, $\cB$ is not unconditional in $KT(p,r)$ if $r>1$.
\end{lemma}

\begin{proof}
Consider $\bx=\sum_{n=1}^{2N} \frac{(-1)^n}{n^{1/p}}\be_n$, with $N\geq 1$. Then,
\[ \Ts\tri{\bx}=\left(\sum_{n=1}^{2N}\frac{1}{n}\right)^{1/r}\, \approx \,[\ln(N+1)]^{1/r}.\]
 On the other hand, for the set $A=\lbrace 1,2,...,2N\rbrace\cap 2\mathbb{Z}$, with cardinality $N$, then,
\[\Ts\tri{P_A(\bx)}\geq \, \|P_A(\bx)\|_\bp=\sum_{n=1}^N \frac{1}{2n}\, \approx\, \ln(N+1).\]
Thus,
$
K_N\geq {\vertiii{
P_A(\bx)}}/{\vertiii{\bx}} \gtrsim \,[\ln(N+1)]^{1/r'}\,
$.
\end{proof}

\begin{lemma}\label{qg3}
	For all $1\leq r\leq \infty$, the space $KT(1,r)$ satisfies $g_N\gtrsim (\ln (N+1))^{1/r'}$.
	In particular, $\cB$ is not quasi-greedy in $KT(1,r)$ if $r>1$.
\end{lemma}
\begin{proof}
For fixed $n\geq1$, consider
$$\bx=\Big(1,\underbrace{-\frac{1}{2^n},...,-\frac{1}{2^n}}_{2^n\text{elements}}, \frac{1}{2}, \frac{1}{2}, \underbrace{-\frac{1}{2^{n+1}},...,-\frac{1}{2^{n+1}}}_{2^{n+1} \text{elements}},...,\frac{1}{2^n},...,\frac{1}{2^n},\underbrace{-\frac{1}{2^{2n}},...,-\frac{1}{2^{2n}}}_{2^{2n} \text{elements}},0,\ldots\Big).$$
Then $\Vert \bx\Vert_{b_1} = 1$, and since the decreasing rearrangement of $\bx$ is given by
$$ \Big(1, \frac{1}{2},\frac{1}{2}, \frac{1}{4},\frac{1}{4}, \frac{1}{4},\frac{1}{4},...,\frac{1}{2^{2n}},..., \frac{1}{2^{2n}},0,\ldots\Big),$$
we also have $\|\bx\|_{\ell^{1,r}}\approx[\sum_{j=0}^{2n}(2^jx^*_{2^j})^r]^{1/r}= [2n+1]^{1/r}\approx\tri{\bx}$.

Now, if $N=1+2+...+2^n = 2^{n+1}-1$, then
$$G_N(\bx)=\Big(1,0,...0,\frac{1}{2},\frac{1}{2},0,...,0,...,\frac{1}{2^n},...,\frac{1}{2^n}\Big),$$
and therefore $\Vert G_N(\bx)\Vert_{b_1} = n+1$. Hence, $\vertiii{G_N(\bx)}\geq n+1 =\log_2(N+1)$, and we conclude
\[
g_N\,\geq\, \tri{G_N(\bx)}/\tri{\bx}\,\gtrsim\,(n+1)^{1/r'}\approx [\ln(N+1)]^{1/r'}.
\]
\end{proof}

\Proofof{Theorem \ref{th_KT}}
From Lemmas \ref{etaxi} and \ref{DN_KTpr} one easily obtains that
\Be
\OT_N(D,D^*)\approx \left\{\Ba{ll}\ln (N+1) & \mbox{ if $1<p<\infty$}\cr
 [\ln (N+1)]^{1/r'} & \mbox{ if $p=1$}.\Ea\right.
\label{OTN_KT}\Ee

\medskip

%
%
%
%
%

When $p=1$, this and Theorem \ref{Th3} give all the upper bounds asserted in \eqref{LNKTpii}.
Since $\min\{\LN,\tLN,K_N\}\gtrsim g_N$ and $\min\{\LN^*,\tLN^*\}\geq \mu^*_N$, the lower bounds
follow from Lemmas \ref{DN_KTpr} and \ref{qg3}.

When $1<p<\infty$, observe from Lemmas  \ref{DN_KTpr} and \ref{qg1} that $\cB$ is quasi-greedy and bidemocratic,
hence also $\cB^*$ must be quasi-greedy, by \cite[Theorem 5.4]{DKKT}.
By \cite[Theorem 3.3]{DKKT}, $\tLN\approx \tLN^*\approx 1$,
as asserted in \eqref{LNKTp}. Also $\LN\approx \LN^*\approx K_N$, by \cite[Theorem 1.1]{GHO2013}, and hence the lower bounds on the left side of \eqref{LNKTp} follow from Lemma \ref{qg2}. It remains to give an upper bound for $K_N$. This time \eqref{OTN_KT} would only be optimal for $r=\infty$. However, for $r<\infty$ we can do slightly better using the fact that $KT(p,r)\hookrightarrow \ell^{p,r}$. Indeed, going back to \eqref{first} in the proof of Theorem \ref{Th3}, first notice that we can choose the sequence $\eta_1(N)=\sum_{j=1}^N1/j^{1/p'}$ because of \eqref{KTD_aux1}.  Then
\Beas
\tri{P_A(\bx)} & \leq & \sum_{j=1}^N a^*_j(x)\Dt\eta_1(j)= \sum_{j=1}^{N} x_j^*\frac{j^{1/p}}{j} \\
&\leq& \Big(\sum_{j=1}^{N} (x_j^*j^{1/p})^r \frac{1}{j}\Big)^{1/r} \Big(\sum_{j=1}^{N}\frac{1}{j}\Big)^{1/r'}\,
\,\leq\, \vertiii{\bx}\, [\ln(N+1)]^{1/r'}.
\end{eqnarray*}
This gives a direct bound $K_N \leq [\ln(N+1)]^{1/r'}$, and completes the proof of the theorem.
\vskip0.3cm
\ProofEnd

\medskip

\bibliographystyle{plain}

\begin{thebibliography}{1}
	
\bibitem{AA2016}
\textsc{F. Albiac, J.L. Ansorena},
\emph{Lorentz spaces and embeddings induced by almosts greedy bases in Banach spaces}
Constr. Approx, 43 (2016), 197--215.
	

	

\bibitem{BBG2016}
\textsc{P.M. Bern\'a, O. Blasco, G. Garrig\'os},
\emph{Lebesgue inequalities for greedy algorithm in general bases},
Rev. Mat. Complut.  {\bf 30} (2017), 369--392.

\bibitem{Bed2008}
\textsc{W. Bednorz},
\emph{Greedy type bases in Banach spaces},
Advances in Greedy Algorithms, Book edited by: W. Bednorz,
November 2008, I-Tech, Vienna, Austria, 325 -- 356 (Open Access Database: www.intechweb.org)








\bibitem{CRS2007}
\textsc{M.J. Carro, J. Raposo, J. Soria}, \emph{Recent developments in the theory of Lorentz spaces and weighted inequalities}. Memoirs Amer. Math. Soc.
877 (2007).

\bibitem{DKK}
\textsc{S.J. Dilworth, N.J. Kalton, D. Kutzarova}, \emph{On the existence of almost greedy bases in Banach spaces},  Studia Math. 159 (2003), no. 1, 67--101.



\bibitem{DKKT}
\textsc{S.J. Dilworth, N.J. Kalton, D. Kutzarova, and V.N.
Temlyakov}, \emph{The Thresholding Greedy Algorithm, Greedy Bases,
and Duality}, Constr. Approx. {\bf 19} (2003), 575--597.

\bibitem{DKO2015}
\textsc{S.J. Dilworth,  D. Kutzarova, T. Oikhberg}, \emph{Lebesgue constants for the weak greedy algorithm}, Rev. Matem. Compl. 28 (2) (2015), 393--409.

\bibitem{DST}
\textsc{S.J. Dilworth, M. Soto-Bajo, V.N. Temlyakov},
\emph{Quasi-greedy bases and Lebesgue-type inequalities}. Studia
Math {\bf 211} (2012), 41--69.


\bibitem{DM}
\textsc{S.J. Dilworth, D. Mitra}, \emph{A conditional quasi-greedy basis
of $\ell^1$}, Studia Math. 144 (2001), 95-100.

\bibitem{Fabian}
\textsc{M. Fabian, P. Habala, P. Hajek, V. Montesinos Santaluc\'\i a, J. Pelant, and V. Zizler},
\emph{ Functional analysis and infinite-dimensional geometry},
Springer-Verlag, New York,  2001. 


\bibitem{GHO2013}
\textsc{G. Garrig\'os, E. Hern\'andez, T. Oikhberg},
\emph{Lebesgue-type inequalities for quasi-greedy bases}, Constr. Approx. 38 (3) (2013), 447--470.

\bibitem{GHN}
\textsc{G. Garrig\'os, E. Hern\'andez, M. de Natividade},
\emph{Democracy functions and optimal embeddings for approximation
spaces}, Adv. Comput. Math.  {\bf 37} (2) (2012), 255-283.


\bibitem{Ga}
\textsc{D.J.H. Garling},
\emph{On symmetric sequence spaces}, Proc. London Math. Soc. (3) {\bf 16} (1966),
85--105.


\bibitem{GN}
\textsc{R. Gribonval, M. Nielsen}, \emph{Some remarks on non-linear
approximation with Schauder bases}, East. J. of Approximation, 7(2),
(2001), 1--19.

\bibitem{Hajek}
\textsc{P. Hajek, V. Montesinos-Santaluc\'\i a, J. Vanderwerff, V. Zizler},
\emph{ Biorthogonal systems in Banach spaces},
SpringerVerlag 2008.

\bibitem{H2011}
     \textsc{E. Hern\'andez},
     \emph{Lebesgue-type inequalities for quasi-greedy bases}.
     Preprint 2011. ArXiv: 1111.0460v2 [matFA] 16 Nov 2011.

\bibitem{HR1970}
\textsc{J.R. Holub, J.R. Retherford},
\emph{Some curious bases for $c_0$ and $C[0,1]$.}
Studia Math., 34 (1970), 227 -- 240.





\bibitem{KT}
\textsc{S.V. Konyagin, V.N. Temlyakov}, \emph{A remark on greedy
approximation in Banach spaces}, East. J. Approx. 5, (1999),
365--379.

\bibitem{KPS}
\textsc{S. Krein, J, Petunin and E. Semenov}, \emph{Interpolation of
Linear Operators}, Translations Math. Monographs, vol. 55, Amer.
Math. Soc., Providence, RI, (1992).






\bibitem{L}
\textsc{J. Lindenstrauss}, \emph{On a certain subspace of $\ell^1$.}
Bull. Acad. Polon. Sci. {\bf 12} (1964), 539-542.

\bibitem{LZ}
\textsc{J. Lindenstrauss, L. Tzafriri}, \emph{Classical Banach
spaces}, vol I, Springer-Verlag 1977.





\bibitem{Osw2001}
\textsc{P. Oswald}, \emph{Greedy algorithms and best
	m-term approximation with respect to biorthogonal systems}, J. Fourier Analysis and Appl., \textbf{7} (4) (2001), 325--341 .

\bibitem{Sin}
\textsc{I. Singer}, \emph{Bases in Banach Spaces}, vol. I,
Springer-Verlag 1970.

\bibitem{SW1971}
\textsc{E.M. Stein, G. Weiss} \emph{Introduction to Fourier Analysis
on Euclidean spaces}, Princeton University Press, Princeton, New
Jersey, 1971.



\bibitem{Tem98trig}
\textsc{V. N. Temlyakov}, \emph{Greedy algorithm and n-term
trigonometric approximation}, Const. Approx., 14, (1998), 569--587.

\bibitem{Tem98haar}
\textsc{V. N. Temlyakov}, \emph{The best $m$-term approximation
and greedy algorithms}. Adv. Comput. Math. 8 (1998), 249--265.

\bibitem{Tem98mhaar}
\textsc{V. N. Temlyakov}, \emph{Nonlinear $m$-term approximation
with regard to the multivariate Haar system}, East J. Approx., 4,
(1998), 87--106.


\bibitem{T2}
\textsc{V.N. Temlyakov}, \emph{Greedy approximation}, Cambridge
University Press, 2011.

\bibitem{Tem2015}
\textsc{V.N. Temlyakov}, \emph{Sparse approximation with bases}.  Advanced courses in Mathematics, CRM Barcelona. Birkh\"auser, 2015.




\bibitem{TYY2011b}
\textsc{V. N. Temlyakov, M. Yang, P. Ye}, \emph{Lebesgue-type
inequalities for greedy approximation with respect to quasi-greedy
bases}, East J. Approx {\bf 17} (2011), 127--138.

\bibitem{Wo} \textsc{P. Wojtaszczyk},
\emph{Greedy Algorithm for General Biorthogonal Systems}, Journal of
Approximation Theory, 107, (2000), 293--314.

\bibitem{Wo3} \textsc{P. Wojtaszczyk},
\emph{Greedy type bases in Banach spaces}.
Constructive theory of functions, 136--155, DARBA, Sofia, 2003.




\end{thebibliography}

\vskip 1truemm

\end{document}